\documentclass{article}
\usepackage{amsmath}
\usepackage{amsfonts}
\usepackage{amssymb}
\usepackage{graphicx}
\graphicspath{ {./images/} }

\usepackage{amsthm}
\usepackage{enumerate}

\begin{document}

\newtheorem{Thm}{Theorem}[section]
\newtheorem{Ass}{Assumption}
\newtheorem{Lem}[Thm]{Lemma}
\newtheorem{Prop}[Thm]{Proposition}
\newtheorem{Cor}[Thm]{Corollary}
\newtheorem{Def}[Thm]{Definition}
\newtheorem{Ex}[Thm]{Example}
\newtheorem{Rem}[Thm]{Remark}

\newcommand{\hr}{\mathbb{H}^2 \times \mathbb{R}}
\newcommand{\la}{\left\langle} 
\newcommand{\ra}{\right\rangle} 
\newcommand{\e}{\epsilon}
\newcommand{\tM}{\widetilde{M}}
\newcommand{\tg}{\tilde{g}}
\newcommand{\tsff}{\tilde{h}}
\newcommand{\tp}{\tilde{p}}
\newcommand{\tG}{\widetilde{\Gamma}}
\newcommand{\tN}{\widetilde{\nabla}}
\newcommand{\tQ}{\widetilde{Q}}
\newcommand{\tL}{\widetilde{L}}
\newcommand{\ok}{\overline{\kappa}}
\newcommand{\uk}{\underline{\kappa}}
\newcommand{\csl}{\partial_{\theta}^2 + \alpha^{-2}E}
\newcommand{\R}{\mathbb{R}}

\title{Constant Mean Curvature $\frac{1}{2}$ Surfaces in $\hr$}
\author{Murray Christian}
\date{October 2019}

\maketitle

\begin{abstract}
This thesis lies in the field of constant mean curvature (cmc) hypersurfaces and specifically cmc $1/2$ surfaces in the three-manifold $\hr$. The value $1/2$ is the critical mean curvature for $\hr$, in that there do no exist closed cmc surfaces with mean curvature $1/2$ or less. Daniel and Hauswirth have constructed a one-parameter family of complete, cmc $1/2$ annuli that are symmetric about a reflection in the horizontal plane $\mathbb{H}^2 \times \{0\}$, the \emph{horizontal catenoids}. In this thesis we prove that these catenoids converge to a singular limit of two tangent horocylinders as the neck size tends to zero. We discuss the analytic gluing construction that this fact suggests, which would create a multitude of cmc $1/2$ surfaces with positive genus.

The main result of the thesis concerns a key step in such an analytic gluing construction. We construct families of cmc $1/2$ annuli with boundary, whose single end is asymptotic to an end of a horizontal catenoid. We produce these families by solving the mean curvature equation for normal graphs off the end of a horizontal catenoid. This is a non-linear boundary value problem, which we solve by perturbative methods. To do so we analyse the linearised mean curvature operator, known as the Jacobi operator. We show that on carefully chosen weighted H\"older spaces the Jacobi operator can be inverted, modulo a finite-dimensional subspace, and provided the neck size of the horizontal catenoid is sufficiently small. Using these linear results we solve the boundary value problem for the mean curvature equation by a contraction mapping argument. 
\end{abstract}

\newpage

\section*{Acknowledgements}

I would especially like to thank my supervisor Jesse Ratzkin for all the mathematical advice, optimism and encouragement, and for his patience during the generous gestation of this thesis. Thanks to Rob Kusner for his interest in this project and the discussions about it. This thesis has benefitted from his remarks and those of the other examiners, Karsten Grosse-Brauckmann and Laurent Hauswirth.

I gratefully acknowledge the financial support of an Innovation Scholarship from the National Research Foundation.

\newpage 

\tableofcontents

\newpage

\section{Introduction}

We begin this thesis with a brief discussion of the field of minimal and constant mean curvature surfaces in the classical setting of $\R^3$. One of the central topics in the field has been the construction and classification of complete, embedded, non-compact examples of such surfaces. We focus on the rather general gluing techniques that have been brought to bear on the construction problem. A central component of these is the analytic construction of spaces of ends, by solving boundary value problems for the non-linear mean curvature equation. The main result of our thesis is a result of this type, in the context of constant mean curvature $1/2$ surfaces in $\hr$. In section \ref{cmchtwor} we discuss some of the recent developments in the theory of these surfaces, before stating our main result, giving an overview of the thesis, and mentioning some future work.  

\subsection{Minimal and constant mean curvature surfaces in $\R^3$}

\subsubsection*{Mean curvature and least area}

Minimal and constant mean curvature surfaces are solutions to the problem of finding surfaces of least area. The classical Plateau problem asks one to find a surface having the smallest possible area amongst all surfaces with boundary a prescribed closed curve. A minimiser for this problem has vanishing mean curvature, and in general one calls a surface with mean curvature zero at all points a minimal surface. On the other hand, surfaces that minimise area subject to an enclosed volume constraint must have constant, non-zero mean curvature. These \emph{minimal} and \emph{constant mean curvature} (cmc) surfaces represent idealised models for soap films and soap bubbles, respectively.

Minimal and cmc surfaces are in many ways distinguished objects. One indication of this is the fact that, locally, they are solutions to a quasilinear, elliptic, partial differential equation (pde). More precisely, if one parametrises such a surface locally as a graph over its tangent plane, the graphing function $u$ satisfies the pde
\begin{equation}  div \left(\frac{\nabla u}{\sqrt{1 + |\nabla u|^2}} \right) = 2H, \label{MCEforgraphs} \end{equation}
where $H$ is the mean curvature of the surface. Such a pde enjoys a maximum principle, with the following striking geometric consequence. Suppose two surfaces having the same constant mean curvature meet tangentially at a point in such a way that their unit normals coincide, and, locally, one surface lies entirely on one side of the other. Then the surfaces coincide in a neighbourhood of that point. In fact, because solutions to \eqref{MCEforgraphs} are real-analytic, the surfaces are identical if they're connected. 

\subsubsection*{Complete, embedded, non-compact examples}

As a consequence of the maximum principle, it's not hard to see that there cannot exist closed (i.e. compact boundaryless) minimal surfaces in $\R^3$. Indeed supposing such a surface existed, one could take a plane disjoint from the surface and parallel translate it until it makes a first, tangential point of contact with the minimal surface. Since the plane has vanishing mean curvature, the argument of the previous paragraph implies that the plane and our putative compact surface are identical, a contradiction.On the other hand, the solution to Plateau's problem by Douglas and Rado guarantees that there is an abundance of compact minimal surfaces with boundary in $\R^3$. 

The minimal surfaces we'll be most interested in are the complete, embedded, non-compact surfaces which have finite topology, in the sense that they are homeomorphic to a compact surface with finitely many punctures. The central examples, apart from flat planes, are the catenoid and helicoid. The catenoid is a surface of revolution, homeomorphic to a sphere with two punctures. The helicoid is ruled by lines, and homeomorphic to a once-punctured sphere. The helicoid and catenoid were discovered in the $1700$'s, and for a long time it was not known if there exist complete, embedded minimal surfaces of finite topology and positive genus.  Only in the $1980$'s did Costa discover such a surface, with genus one and three ends. This and subsequent examples with higher genus have been constructed using the Enneper-Weierstrass representation. This representation produces minimal surfaces from complex-analytic data, but is peculiar to mean curvature zero, and other methods are required for other cmc surfaces. 


As in the case of minimal surfaces, there are rather limited possibilities for closed, embedded cmc surfaces in $\mathbb{R}^3$. In fact the only possibilities are round spheres, as was shown by Alexandrov using his ingenious method of moving planes. The argument used above to show that there are no closed minimal surfaces in $\mathbb{R}^3$ is a moving planes argument, and illustrates some of the general features, particularly the use of a maximum principle for an elliptic pde. As far as complete examples are concerned, the most important embedded cmc surfaces are the Delaunay unduloids. These are a one-parameter family of rotationally invariant, periodic  surfaces that interpolate between a cylinder and a singular limit of tangentially intersecting spheres whose centers lie along a line. The most powerful method for producing new examples of embedded cmc surfaces has so far proved to be analytic gluing techniques, which were pioneered in this setting by Kapouleas \cite{Kap1}, \cite{Kap}. We turn to these now.

\subsubsection*{Gluing and spaces of ends}

The Delauney unduloids and catenoids play much deeper roles in the theory of cmc and minimal surfaces than simply that of examples. This is because they form the asymptotic models for embedded annular ends of their respective types of surfaces. Indeed Schoen \cite{Sch83} proved that every embedded, minimal end is asymptotic to a catenoid or a plane, and it was later shown that every embedded annular cmc end is asymptotic to the end of a Delaunay surface \cite{KKS}. Therefore, one knows that any complete embedded minimal or cmc surface one may construct will have ends that are perturbations of these asymptotic models. This suggests a procedure for building new examples by attaching perturbations of these model ends to central building blocks. 

One variant of this gluing procedure is the method of Cauchy data matching developed by Mazzeo, Pacard, Pollack and others. In \cite{MP} the authors construct new examples of complete, embedded cmc surfaces by attaching Delaunay ends to certain truncated minimal surfaces called $k$-noids. Since this nicely illustrates the method we describe it in some detail. 

To begin with we'd like to give some motivation for why such a construction may be possible. The crucial ingredient is the one-parameter family of minimal catenoids: mean curvature scales like $\e^{-1}$ under a homothety by $\e > 0$, thus all homotheties of the catenoid $\mathcal{C}$ are also minimal surfaces. The parameter $\e$ is the necksize of the corresponding catenoid $\e \mathcal{C}$. The importance of this family is that it provides asymptotic models both for (non-planar) ends of $k$-noids, \emph{and} for neck regions of Delaunay unduloids.  
The compact, central building blocks in \cite{MP} are truncations of $k$-noids with asymptotically catenoidal ends. To each end one associates a necksize parameter $a_i$ which is that of the catenoid to which that end is asymptotic. Similarly any Delaunay surface has a necksize parameter $\e$, and as $\e \rightarrow 0$, neighbourhoods of the neck regions are asymptotic to the catenoids $\e \mathcal{C}$. Thus if one shrinks the $k$-noid sufficiently, each of its ends may be closely matched to a neck region of the unduloid that shares the same catenoidal model, or in other words, the same necksize parameter. 

The method of Cauchy data matching consists of producing families of cmc $1$ surfaces close to each of the building blocks in the construction. 
This is done by solving the Dirichlet problem for the mean curvature equation for normal graphs off these pieces. The family is then parametrised by the boundary data one can prescribe for this problem. Since one is only interested in surfaces close to the given models, the solution of the non-linear equation is perturbative, after inverting the linearisation of the mean curvature equation. The next step is then to show that it is possible to find a member of each family such that, at each interface, the boundaries match to zeroth and first order. The  resulting surface is then a continuously differentiable weak solution to the cmc $1$ equation, and hence by elliptic regularity, a smooth cmc $1$ surface. 

\subsection{Minimal and cmc $1/2$ surfaces in $\hr$} \label{cmchtwor}


Whilst many deep open problems remain in the study of cmc surfaces in $\R^3$, the field has also expanded to other ambient three-manifolds, in particular the eight homogeneous geometries of Thurston. To put these in context, recall that in two dimensions, the constant curvature geometries $\R^2, \mathbb{S}^2$ and $\mathbb{H}^2$ play something of a universal role. Indeed, every compact, orientable two-manifold is homeomorphic to a quotient of one of these spaces by a discrete group of isometries. In other words, each of these \emph{topological} spaces admits a canonical \emph{geometry}. The famous Geometrisation Conjecture of Thurston asserts that something analogous is true in three dimensions; of course, this conjecture has now been proven through the work of Hamilton and Perelman. Loosely speaking, the Geometrisation Theorem asserts that every compact, orientable three-manifold can be topologically decomposed in such a way that each piece admits a canonical geometry. More precisely, each piece is homeomorphic to the quotient of a 'model' geometry by a discrete group of isometries. Unlike the two-dimensional case where the models have constant curvature, there are eight model geometries, which constitute all the simply-connected, homogeneous three-manifolds. They are

\[ \R^3, \hspace{5pt} \mathbb{S}^3, \hspace{5pt} \mathbb{H}^3, \hspace{5pt} \mathbb{S}^2 \times \R, \hspace{5pt} \mathbb{H}^2 \times \R, \hspace{5pt} \widetilde{SL_2(\R)}, \hspace{5pt} Nil_3 \mbox{ and } \hspace{1pt} Sol. \]

The constant curvature spaces $\R^3, \mathbb{S}^3$ and $\mathbb{H}^3$ possess the maximum possible symmetry, their isometry groups are six-dimensional. The isometry groups of the product spaces $\mathbb{S}^2 \times \R$ and $\mathbb{H}^2 \times \R$ are four-dimensional, as are those of the universal cover $\widetilde{SL_2(\R)}$ of $SL_2(\R)$, and the Heisenberg group $Nil_3$. The least symmetric of the eight model geometries is the Lie group $Sol$, which has a three-dimensional isometry group. See \cite{S} for further details. 


Before we focus on $\hr$ we would like to mention two highlights of the theory of cmc surfaces in the Thurston geometries, which extend results on cmc surfaces in space forms to the homogenous manifolds with four-dimensional isometry group. One of the first such results was the discovery by Abresch and Rosenberg \cite{AR} of a holomorphic quadratic differential on cmc surfaces in $\mathbb{S}^2 \times \R$ and $\hr$, which generalises the classical Hopf differential for cmc surfaces in $\R^3$ (and the other space forms). The most famous application of Hopf's differential is the classification of immersed cmc spheres as round spheres. Using their holomorphic quadratic differential, \cite{AR} classified the immersed cmc spheres in $\mathbb{S}^2 \times \R$ and $\hr$: they are all embedded and rotationally invariant. For a survey of other global results in these spaces, including analogs of Alexandrov's classification of compact, embedded cmc surfaces, and Bernstein's classification of entire minimal graphs, see \cite{FM}.

Another important result, due to Daniel \cite{Dan}, is the solution of the isometric embedding problem for surfaces in three-dimensional manifolds with four-dimensional isometry group. In the space forms the Gauss and Codazzi equations are necessary and sufficient conditions for a surface to be locally isometrically embedded with prescribed shape operator. Due to the reduced symmetry in  $\mathbb{S}^2 \times \R, \mathbb{H}^2 \times \R, \widetilde{SL_2(\R)}$ and $Nil_3$, the Gauss and Codazzi equations must be augmented by two further equations relating to the structure of these spaces as fibrations over a two-dimensional space form. Using these integrability conditions, Daniel generalised the Lawson correspondence between cmc surfaces in the space forms to this setting. For example, under Daniel's correspondence every simply-connected, immersed, minimal surface in $Nil_3$ is isometric to a simply-connected, immersed, cmc $1/2$ surface in $\hr$, and a conformal parametrisation of one of these gives rise to a conformal parametrisation of the other.  



\subsubsection*{Minimal surfaces in $\hr$}

One of the pioneering works on cmc surfaces in $\hr$ was that of Nelli and Rosenberg \cite{NR}, in which they produced many examples of complete minimal surfaces. These included a family of rotationally invariant annuli, or catenoids, and a family of screw motion invariant surfaces generated by a horizontal geodesic, that is, helicoids. They proved that, given any $\mathcal{C}^1$ vertical graph in the asymptotic boundary $\partial_{\infty} \mathbb{H}^2 \times \mathbb{R}$, there is a unique, entire minimal graph asymptotic to the given curve. They also proved a Jenkins-Serrin theorem, showing that over suitable domains there exist minimal graphs with infinite boundary data on some portions of the boundary, and arbitrary continuous boundary data on others. Lastly, in conjunction with the work of Collin and Rosenberg \cite{CR}, many Scherk-type surfaces have been constructed. These are minimal graphs over domains bounded by finite polygons or ideal polygons with $2n$ sides, which have asymptotic boundary values alternating between $\infty$ and $-\infty$ on the edges of the polygon. 

In addition to the helicoids just mentioned, Sa Earp and Toubiana \cite{S-T05} constructed surfaces invariant under a screw motion but whose generating curve is not a horizontal geodesic. In \cite{S-T08} these authors also constructed disc-type minimal surfaces invariant under hyperbolic translations. 

The last examples of minimal surfaces we wish to mention are the positive genus surfaces constructed using gluing techniques by Martin, Mazzeo and Rodriguez \cite{MMR}. The key building blocks in their construction are the horizontal minimal catenoids constructed independently by Pyo \cite{Pyo} and  Morabito and Rodriguez \cite{MoR}. These embedded minimal annuli have ends asymptotic to vertical geodesic planes, and finite total curvature. In fact, Hauswirth, Nelli, Sa Earp and Toubiana \cite{HNST} proved that these are the unique complete, immersed, finite total curvature, minimal surfaces having two ends asymptotic to vertical planes; a result analogous to Schoen's \cite{Sch83} characterisation of minimal catenoids in $\mathbb{R}^3$. 

The gluing procedure of \cite{MMR} relies on the geometry of the ends of the horizontal catenoids through the following key observation. If two horizontal catenoids each have an end asymptotic to a common vertical plane, and if the distance between their necks is sufficiently large, then these two ends will be very close to each other, and an approximately minimal surface can be constructed by gluing the ends together with a small cutoff function.  

Now consider a collection $\mathcal{G}$ of disjoint geodesics in $\mathbb{H}^2$, and a collection $\mathcal{A}$ of geodesic arcs connecting pairs of geodesics in $\mathcal{G}$ in such a way that each arc in $\mathcal{A}$ realises the distance between the pair it connects. Replace each arc in $\mathcal{A}$ with a horizontal catenoid that is asymptotic to the vertical planes over the geodesics in $\mathcal{G}$ that it connects. Using cut-off functions to glue these catenoids together, they create an approximately minimal surface. If the catenoids' necks are all sufficiently far away from each other, then the approximately minimal surface can be perturbed to a minimal surface. By considering different configurations of geodesics $\mathcal{G}$ and connecting arcs $\mathcal{A}$, the authors construct surfaces of any positive genus. In order to  keep the neck separation sufficiently large, they require many ends for each handle they attach.




\subsubsection*{Constant mean curvature $1/2$ surfaces in $\hr$.}

Apart from minimal surfaces, there is a second distinguished class of cmc surfaces in $\hr$, those with cmc $1/2$. The value $1/2$ is the critical mean curvature of $\hr$, that is, the largest value of the mean curvature for which there are no closed cmc surfaces. This can be seen by an argument similar to the one for minimal surfaces in $\mathbb{R}^3$, once one has the appropriate analogue of a plane, namely a horocylinder. A horocylinder is a vertical cylinder over a horocycle in $\mathbb{H}^2$. Since horocycles have constant curvature $1$, the horocylinders have constant mean curvature $1/2$. Like parallel planes in $\mathbb{R}^3$, horocylinders foliate $\hr$. Now, if there was a closed cmc $1/2$ surface in $\hr$, then there would exist a horocylinder that makes one-sided, tangential contact with it, with coinciding normal vector. Once again, the geometric maximum principle implies that the surfaces coincide, a contradiction. On the other hand, for every $H > 1/2$ there do exist closed cmc $H$ surfaces, for example the rotational spheres. 

There are many examples of complete cmc $1/2$ surfaces in $\hr$. These include screw motion invariant surfaces, and both embedded and immersed rotationally invariant annuli that are symmetric with respect to reflection in a horizontal plane \cite{S-T05}. Another important example is the hyperboloid, which is an entire, rotationally invariant vertical graph whose height function tends to infinity on the end of the surface. Nelli and Sa Earp \cite{NS} proved a half-space theorem with respect to the hyperboloid: the only complete, properly immersed cmc $1/2$ surfaces that lie on the mean convex side of a hyperboloid are vertical translations of the hyperboloid. A different half-space theorem was proved by Hauswirth, Rosenberg and Spruck \cite{HRS}, with respect to horocylinders. Their result is that a complete, properly immersed cmc $1/2$ surface contained in the mean convex side of a horocylinder must be another horocylinder.

A fair amout is known about vertical graphs of cmc $1/2$. Hauswirth, Rosenberg and Spruck \cite{HRS09} proved a Jenkins-Serrin theorem for them, and \cite{ENS} constructed examples by solving exterior Dirichlet problems, obtaining surfaces whose ends are contained between two hyperboloids. Cartier and Hauswirth \cite{CH} showed that it is possible to pertub the ends of the hyperboloids to obtain entire graphs whose horizontal distance from a hyperboloid in a slice $\mathbb{H}^2 \times \{t\}$ and direction $\theta$ is given, in the limit as $t \rightarrow \infty$, by an essentially arbitrary, small function of $\theta$.  Entire examples have been constructed using the harmonic Gauss map of Fernandez and Mira \cite{FM07} and a result of \cite{HRS}, and the entire examples have been completely classified (see \cite{FM}). 

Moving away from vertical graphs, Plehnert has constructed examples with dihedral symmetry using a conjugate Plateau construction \cite{P1}, \cite{P2}. This method proceeds by solving Plateau problems for minimal surfaces in the Heisenberg group $Nil_3$, then using the sister correspondence of Daniel \cite{Dan} to obtain cmc $1/2$ surfaces in $\hr$, which are extended by Schwarz reflection. Genus one surfaces with $k$-ends, and genus zero $k$-noids have been constructed this way. 

Lastly we mention the cmc $1/2$ surfaces that are the topic of this thesis. Using the Gauss map of \cite{FM07}, Daniel and Hauswirth  constructed a family of complete, properly embedded cmc $1/2$ annuli whose axes lie in a horizontal plane; the horizontal catenoids \cite{DH}. The family is parametrised by their necksizes. We discuss their geometry and the small necksize limit in detail in section \ref{DHshoricats}.








\subsection{Main results and outline of the thesis}

The main result of this thesis is the construction of families of cmc $1/2$ annular ends asymptotic to the ends of horizontal catenoids. Let $\Sigma_{\e}$ denote the horizontal catenoid of necksize $\e$, and $\Sigma_{\e}^c$ the truncation of the end (to be defined precisely later). Stated somewhat imprecisely, our result is

\begin{Thm}
There exists $\e_0 >0$ such that for all $\e \in (0, \e_0)$, there is a projection
$\pi''_{\e}$ on $\mathcal{C}^{2,\alpha}(\partial \Sigma_{\e}^c)$ with two-dimensional kernel, and the following property. For each $\phi \in \pi''_{\e}\mathcal{C}^{2,\alpha}(\partial \Sigma_{\e}^c)$ there is a unique decaying solution to the boundary value problem for the mean curvature operator
\[ \left \{ \begin{array}{cc} \mathcal{M}_{\e}(w) = 1/2 & \mbox{ on } \Sigma_{\e}^c \\ \pi''_{\e}w = \phi & \mbox{ on } \partial \Sigma_{\e}^c \end{array} \right. \]
\end{Thm}

In Section $2$ we introduce two models of $\hr$, and Daniel and Hauswirth's original parametrisation of the horizontal catenoids. We then map this to the upper half plane model of $\hr$, and reparametrise the catenoids with the coordinates that we will perform our analysis with. 

Our approach to constructing the families of cmc $1/2$ ends is perturbative. In section $3$ we expand the mean curvature operator on normal graphs in the form
\[ \mathcal{M}_{\e}(w) = 1/2 + \mathcal{J}_{\e}w + \mathcal{Q}_{\e}(w). \]
The linearisation $\mathcal{J}_{\e}$ is the well-known Jacobi operator, and our primary focus is to understand the structure of the non-linear operator $\mathcal{Q}_{\e}$ which collects the quadratic and higher order terms in $w$ and its derivatives. In particular we show that the coefficients of $\mathcal{Q}_{\e}$ are functions that are bounded on the ends of the catenoids, independently of $\epsilon$. To obtain these results we begin with an expansion of the mean curvature equation in a more general setting using Fermi coordinates. Then we incorporate information on the geometry of $\hr$ and the horizontal catenoids to establish the final estimates. 

Section $4$ is devoted to analysing the linear Jacobi operator. We begin by calculating an explicit expression for the operator in our chosen coordinates, which brings out a close similarity to the Jacobi operator on minimal catenoids in Euclidean space. This similarity allows us to obtain explicit expressions for several Jacobi fields that arise from geometric deformations of the horizontal catenoids.  Our two main linear results provide solutions of homogeneous and inhomogeneous boundary value problems for the Jacobi operator with estimates. This is done in the setting of weighted H\"older spaces, using eigenfunction expansions. 

With the linear analysis in place, Section $5$ contains the solution of the boundary value problems for the mean curvature operator, by reformulating it as a fixed point problem. Our results on the structure of the non-linear term allow us to show that for small neck-sizes, the relevant operator is a contraction on a suitable small ball, proving existence of the desired cmc surfaces.

The final Section $6$ outlines a gluing construction for cmc $1/2$ surfaces in $\hr$, in which the space of ends asymptotic to horizontal catenoids would play an important role. We discuss the proposed construction, as well the technical obstacles that have prevented us from completing it.




\newpage

\section{Constant mean curvature $\frac{1}{2}$ horizontal catenoids in $\hr$}

\subsection{Models of $\hr$}

We will use two models for $\hr$, corresponding to the upper half-plane and disc models of the hyperbolic plane. Hauswirth and Rosenberg parametrise the horizontal catenoids in the disc model of $\hr$, but we will map the parametrisation to the upper half plane model and conduct all our analysis there. 

By the upper half plane model of $\hr$ we mean, of course, the upper half-plane model of hyperbolic space crossed with the real line. This is the set $\{(x,y,z) \in \mathbb{R}^3: y > 0\}$ endowed with the product metric $g_{\hr}$, which in these coordinates takes the form
\[ g_{\hr} = \frac{dx^2 + dy^2}{y^2} + dz^2. \]

On the other hand, the disk model is the set $\{(\tilde{x}, \tilde{y},z) \in \mathbb{R}^3: \tilde{x}^2 + \tilde{y}^2 < 1\}$ with the metric 
\[ g_{\hr} = 4\frac{d\tilde{x}^2 + d\tilde{y}^2}{(1 - \tilde{x}^2 - \tilde{y}^2)^2} + dz^2. \]

In both models the coordinate $z$ parametrises the $\mathbb{R}$ factor, which we shall refer to as the \emph{vertical} factor, in contrast to the \emph{horizontal} hyperbolic factor. Similarly, there is an orthogonal direct sum decomposition of the tangent bundle as 
\[ T(\hr) \cong T\mathbb{H}^2 \oplus T\mathbb{R}. \] Vectors lying in the first factor will be called \emph{horizontal}, whilst those in the second factor are \emph{vertical}. The horizontal and vertical components of a tangent vector $X$ we denote by $X_h$ and $X_v$ .  

The isometries we will use to relate the two models act trivially on the vertical factor. On the horizontal factor they are the M\"obius transformations that exchange the ordered triples $(i, 0, -i)$ and $(\infty, i, 0)$ , when we identify the pairs $(x,y)$ and $(\tilde{x}, \tilde{y})$ with the complex variables $\xi = x + iy$ and $\tilde{\xi} = \tilde{x} + i\tilde{y}$ respectively. Explicitly, these maps are given by
\begin{equation} \xi = i \left( \frac{i + \tilde{\xi}}{i - \tilde{\xi}}\right) \mbox{ mapping $(i,0,-i)$ to $(\infty, i, 0)$, and }  \label{balltouhp}\end{equation}
\[ \tilde{\xi} = i\left(\frac{\xi - i}{\xi + i}\right) \mbox{ mapping $(\infty, i, 0)$ to $(i,0,-i)$.} \]

The isometry group of $\hr$ is isomorphic to the product of isometry groups of $\mathbb{H}^2$ and $\mathbb{R}$. The isometry group of the upper half plane model of $\mathbb{H}^2$ is generated by three families of orientation-preserving isometries, 
as well as an orientation-reversing inversion. The orientation preserving families are
parabolic translations
\[ \tau_a(x,y) = (x + a, y), \hspace{5pt} a \in \mathbb{R}, \]
hyperbolic dilations
\[ \delta_a(x,y) = (ax, ay), \hspace{5pt} a > 0, \]
elliptic rotations about $(0,1)$,
\[ \rho_a(x,y) = \frac{\left(\frac{1}{2}(1 - x^2 - y^2)\sin a + x \cos a, y \right)}{(x^2 + y^2)\sin^2(a/2) + \cos^2(a/2) - x\sin a}, \hspace{5pt} a \in [0,2\pi], \]
whereas inversion in the semi-circle $x^2 + y^2 = 1, y > 0$ reverses orientation. Each of the one-parameter families of translations, dilations and rotations give rise to a Killing field, and these are responsible for a certain degeneracy of the linear operators we will study. The hyperbolic plane is homogeneous and isotropic, so the next result is not surprising.

\begin{Lem}\label{seccurv}
The sectional curvature $K_{\hr}(P)$ of a plane $P \subset T_p(\hr)$ is given in terms of a unit normal vector $\nu$ to $P$ by
\[ K_{\hr}(P) = -\la \nu, \partial_z \ra^2 \]
\end{Lem}

\begin{proof}
The product structure of $\hr$ and the fact that $\mathbb{R}$ has zero curvature imply that the curvature tensor only depends on the horizontal components of its arguments. Furthermore the horizontal factor has constant curvature $-1$, so the curvature tensor of $\hr$ is given by
\[ R(X,Y,Z,W) = R(X_h,Y_h,Z_h,W_h) = - (\la X_h, W_h\ra \la Y_h, Z_h \ra - \la X_h, Z_h \ra \la Y_h, W_h \ra). \]
Given the plane $P$, choose an orthonormal basis $X,Y$ for $P$ where $X$ is horizontal; this is always possible as we can choose $X$ orthogonal (in P) to the projection of $\partial_z$ onto $P$. With this choice of basis we have
\[ K_{\hr}(P) = R(X,Y,Y,X) = - (|X_h|^2|Y_h|^2 - \la X_h, Y_h \ra^2) = -|Y_h|^2. \]
Expressing the vertical vector $\partial_z$ in terms of the orthonormal basis $X,Y, \nu$ we have $\partial_z = \la Y, \partial_z \ra Y + \la \nu, \partial_z \ra \nu$, so Pythagorous' theorem implies that
\[ \la \nu, \partial_z \ra^2 = 1 - \la Y, \partial_z \ra^2 = 1 - |Y_v|^2 = |Y_h|^2.  \]
Combining the last two equations gives the result. \qedhere
\end{proof}

\begin{Lem} \label{secplusricci}
Let $\nu$ be a unit normal for $P$. Then $K_{\hr}$ and the Ricci tensor of $\nu$ are related by \[ K_{\hr}(P) + Ric(\nu)  \equiv - 1. \] 
\end{Lem}

\begin{proof}
To see this note that $K_{\hr}(P) + Ric(\nu)$ is the sum of sectional curvatures of three orthogonal planes (specifically $P$ and any two orthogonal planes containing $\nu$). However $K_{\hr} + Ric(\nu)$ is also half the scalar curvature, so the sum is independent of the orthogonal planes chosen. Choosing the three planes spanned by each pair of the coordinate vectors $\partial_x, \partial_y$, $\partial_z$ we obtain the given value. \qedhere
\end{proof}

\subsection{The horizontal catenoids}\label{DHshoricats}

In this section we describe our parametrisation of the horizontal catenoids (Proposition \ref{stparm}). This is  obtained from the original parametrisation of Daniel and Hauswirth by, firstly, mapping it from the ball to the half-plane model of $\hr$, and secondly, making a change of coordinates motivated by the resulting form of the Jacobi operator. Thereafter we discuss their geometry and calculate some explicit formulae for various curvatures that we will use in computing the Jacobi operators. 

The original, conformal parametrization of the horizontal catenoids obtained by \cite{DH} is given in Proposition \ref{DHparm}. We begin by describing the parameters and functions needed to define the parametrisation. For details on the origin of these, see \cite{DH}.

The one-parameter family of catenoids is indexed by $\alpha > 0$. We define equivalent parameters
\begin{equation} \alpha_* := \sqrt{\alpha^2 + 1} \mbox{ and } \epsilon := \frac{\alpha_*}{\alpha} - 1. \label{epsilon} \end{equation}
The length of the shortest closed geodesic on the horizontal catenoid with parameter $\alpha$ is on the order of $\epsilon$, so we may think of $\epsilon$ as a neck-size parameter. We will use both $\alpha$ and $\epsilon$ as convenient. Typically our results hold for all large $\alpha$, or equivalently for all sufficiently small $\e > 0$. 

\begin{Def}\label{aux}
The functions $\varphi(u)$ and $\varphi_*(u)$ are  defined as the solutions to the initial value problems
\begin{equation}
(\varphi')^2 = \alpha^2 + \cos^2 \varphi, \hspace{5pt} \varphi(0) = 0, \varphi'(0) < 0, 
\end{equation}
\begin{equation}
(\varphi_*')^2 = \alpha_*^2 - \cos^2 \varphi_*, \hspace{5pt} \varphi_*(0) = 0, \varphi_*'(0) < 0.
\end{equation}
We also define 
\[ f(u) = \frac{\alpha \cos \varphi_* - \alpha_* \cos \varphi}{\alpha \cos \varphi \cos^2 \varphi_*}. \]
\end{Def}

Now we can state the original parametrization of the horizontal catenoids. 

\begin{Prop}[Horizontal catenoids of \cite{DH}]\label{DHparm}
For each $\alpha > 0$ there is a constant mean curvature $1/2$ horizontal catenoid parametrised in the ball model of $\hr$  by
\[ \widetilde{\mathbf{X}}_{\alpha}(u,v) = \left(\frac{X_1}{1 + X_3}, \frac{X_2}{1 + X_3}, z \right) \hspace{5pt} \mbox{ where } \]
\begin{eqnarray*}
X_1(u,v) & := &  \cosh(\alpha v) \sin \varphi_*(u) f(u) \\
X_2(u,v) & := & \cosh(\alpha v)\sinh(\alpha_*v)\left(f(u) + \frac{\alpha_*}{\alpha}\right) - \cosh(\alpha_*v)\sinh(\alpha v) \\
X_3(u,v) & := & \cosh(\alpha v)\cosh(\alpha_*v)\left(f(u) + \frac{\alpha_*}{\alpha}\right) - \sinh(\alpha_*v)\sinh(\alpha v) \\
z(u,v) & := & \frac{\cos \varphi \cosh(\alpha v)}{\alpha(\alpha - \varphi')}.
\end{eqnarray*}
The metric in these coordinates is
\[ \frac{\cosh^2(\alpha v)}{(\alpha - \varphi')^2}(du^2 + dv^2). \]
\end{Prop}

Before we can map this to the upper half plane we need some preliminary calculations to obtain reasonably simple expressions when we do so.

\begin{Lem}\label{identityvarphi}
The following identities relate $\varphi$ and $\varphi_*$:
\[ -\varphi' \cos \varphi_* = \alpha_* \cos \varphi, \hspace{10pt} -\varphi' \sin \varphi_* = \alpha \sin \varphi. \]
These allow us to replace $\cos \varphi_*$ and $\sin \varphi_*$ with functions of $\varphi$ in Lemma \ref{XID} below. 
\end{Lem}

\begin{proof}
This is proved in Lemma $8.1$ of \cite{DH}. Using the Definition \ref{aux} it is straightforward to verify that the functions $\frac{-\varphi'}{\cos \varphi}$ and $\frac{\alpha_*}{\cos \varphi_*}$ both solve the initial value problem
\[ (A')^2 = (A^2 - 1)(A^2 - \alpha_*^2), \hspace{5pt} A(0) = \alpha_*. \] This gives the first identity, and the second follows from the first using the definitions \ref{aux} and standard trigonometric identities. \qedhere
\end{proof}

\begin{Lem}\label{XID}
The functions $X_1, X_2$ and $X_3$ satisfy the identity
\[ X_3^2 - 1 = X_1^2 + X_2^2. \]
\end{Lem}

\begin{proof}
We'll show that $X_3^2 - X_1^2 - X_2^2 = 1$. Expanding the squares and using hyperbolic trigonometric identities we get
\[ X_3^2 - X_1^2 - X_2^2 = \cosh^2(\alpha v) \left\{ \left(f(u) + \frac{\alpha_*}{\alpha}\right)^2 - \sin^2 \varphi_*(u)f^2(u)\right\} - \sinh^2(\alpha v). \]
Hence the claimed formula will follow once we show that the quantity in $\{ \cdots \}$ is identically one. Direct calculation from the expression for $f$ (see Definition \ref{aux}) shows that this is equivalent to the identity
\[ \alpha_*^2\cos^2 \varphi \sin^2 \varphi_* = \alpha^2 \cos^2 \varphi_* \sin^2 \varphi \]
and the validity of this expression can be seen by using the identities of Lemma \ref{identityvarphi} to replace $\sin^2 \varphi_*$ on the left hand side and $\cos^2 \varphi_*$ on the right hand side. \qedhere
\end{proof}

Now we are in a position to map the parametrisation in Definition \ref{DHparm} to the upper half plane. The isometry we will use acts trivially on the $\mathbb{R}$ factor, so we only need to compute its effect on $(\tilde{x}, \tilde{y})$. In terms of real and imaginary parts, the map \eqref{balltouhp} of the ball model to the upper half plane model is given by
\begin{equation} (x,y) = \left(\frac{2\tilde{x}}{\tilde{x}^2 + (\tilde{y} - 1)^2}, \frac{1 - \tilde{x}^2 - \tilde{y}^2}{\tilde{x}^2 + (\tilde{y} - 1)^2}\right). \label{balltouhpcoord} \end{equation} 
Note that this is \emph{not} the most commonly used map between these models; it is our choice because it places the ''axis'' of the catenoids along the $x = 0$ axis, and the ends at $0 \in \mathbb{R}$ and $\infty$. 
\begin{Lem}
Composing $\widetilde{\mathbf{X}}_{\alpha}$ with the mapping \eqref{balltouhp} to the upper half plane we have  
\[ (x,y) = \left(\frac{X_1}{X_3 - X_2}, \frac{1}{X_3 - X_2} \right). \]
\end{Lem} 

\begin{proof}
Given that $\tilde{x} = \frac{X_1}{1 + X_3}$ and $\tilde{y} = \frac{X_2}{1 + X_3}$, the identity of Lemma \ref{XID} gives
\[\tilde{x}^2 + (\tilde{y} - 1)^2 = \frac{X_1^2 + X_2^2 - 2X_2(1 + X_3) + (1 + X_3)^2}{(1 + X_3)^2} = \frac{2(X_3 - X_2)}{1 + X_3}. \] Another application of the identity gives
\[ 1 - \tilde{x}^2 - \tilde{y}^2 = \frac{(1 + X_3)^2 - (X_1^2 + X_2^2)}{(1 + X_3)^2} = \frac{2}{1 + X_3}. \]
Therefore the claim follows from \eqref{balltouhpcoord}. \qedhere
\end{proof}

\begin{Lem}
The denominators $X_3 - X_2$ are given by
\[ X_3 - X_2 = e^{-(\alpha_* - \alpha)v}\left(1 + e^{-\alpha v}\cosh(\alpha v)\left(f + \frac{\alpha_*}{\alpha} - 1\right)\right). \]
\end{Lem}

\begin{proof}
Expanding the hyperbolic trigonometric functions in Proposition \ref{DHparm} as sums of exponentials we obtain
\[X_3 - X_2 = \frac{e^{-(\alpha_* - \alpha)v}}{2}\left(f + \frac{\alpha_*}{\alpha} + 1 \right) + \frac{e^{-(\alpha_* + \alpha)v}}{2}\left(f + \frac{\alpha_*}{\alpha} - 1 \right), \] from which the claimed formula follows straightforwardly.\qedhere 
\end{proof}

\begin{Lem}
\[ f + \frac{\alpha_*}{\alpha} - 1 = \frac{\sin^2 \varphi}{\alpha_*(\alpha + \alpha_*)(\alpha - \varphi')(\alpha_* - \varphi')}. \]
\end{Lem}

\begin{proof}
By calculation, using Definition \ref{aux} and Lemma \ref{identityvarphi}, we omit the details. \qedhere
\end{proof}

\subsubsection*{New coordinates $s$ and $\theta$}

We define coordinates $(s,\theta)$ on the horizontal catenoids by
\[ s := \alpha v \mbox{ and } \theta := - \varphi(u). \]

Making the change of variables from $(u,v)$ to $(s,\theta)$ and putting together the results above we have the parametrisation of the catenoids in the upper half plane that we will use for our analysis.

\begin{Prop}\label{stparm}
For each $\epsilon > 0$ there is a constant mean curvature $1/2$ horizontal catenoid paramatrised in the upper half-plane model of $\hr$ by
\begin{equation} \mathbf{X}_{\alpha}(s,\theta) = \left(\frac{\sin \theta}{\alpha_*(\alpha - \varphi')}(\cosh s)e^{\epsilon s} \omega(s,\theta), e^{\epsilon s}\omega(s,\theta), \frac{\cos \theta}{\alpha(\alpha - \varphi')} \cosh s\right),     \mbox{ where } \label{parm} \end{equation}
\[ \omega(s,\theta) = \left(1 + \frac{\sin^2 \theta}{\alpha_*(\alpha + \alpha_*)(\alpha - \varphi')(\alpha_* - \varphi')} e^{-s} \cosh s\right)^{-1}. \]
The metric in these coordinates is given by 
\begin{equation} \frac{\cosh^2s}{\alpha^2(\alpha - \varphi')^2}\left( ds^2 + \frac{\alpha^2}{(\varphi')^2} d\theta^2\right). \label{stmetric} \end{equation}
\end{Prop}



\begin{Prop}\label{catsymmetries}
The catenoids are invariant under a finite group of ambient isometries generated by three reflections. Indeed, reflection in the plane $x = 0$ interchanges $\mathbf{X}_{\alpha}(s,\theta)$ and $\mathbf{X}_{\alpha}(s,-\theta)$, and reflection in the plane $z = 0$ interchanges $\mathbf{X}_{\alpha}(s,\theta)$ and $\mathbf{X}_{\alpha}(s,\pi - \theta)$. Inversion in the semi-circle $x^2 + y^2 = 1, y > 0$ interchanges $\mathbf{X}_{\alpha}(s,\theta)$ and $\mathbf{X}_{\alpha}(-s,\theta)$, 
\end{Prop}

\begin{proof}
The first two symmetries follow from the fact that $\omega(s,\theta)$ and $\varphi'(\theta)$ are invariant under $\theta \mapsto - \theta$ and $\theta \mapsto \pi - \theta$. Due to the somewhat complicated formula for inversion in the unit circle, it is more difficult to verify the last symmetry in these coordinates. However, in the original parametrization of the catenoids in the ball model (Proposition \ref{DHparm}), it is easy to see that reflection in the geodesic $\tilde{y} = 0$ interchanges $\widetilde{X}_{\alpha}(u,v)$ and $\widetilde{X}_{\alpha}(u,-v)$, which is exactly the symmetry we claim in our model and coordinates.
\end{proof}

\begin{Lem}\label{seccurvcoord}
Let $K_{\hr}(s,\theta)$ denote the sectional curvature in $\hr$ of the tangent plane to the horizontal catenoids at $\mathbf{X}_{\alpha}(s, \theta)$. Then
\[ K_{\hr}(s,\theta) = -\frac{\cos^2\theta}{\cosh^2 s}. \]
Note the surprising fact that this is independent of $\alpha$. 
\end{Lem}

\begin{proof}
We obtain this formula by expressing $K_{\hr}(s,\theta)$ in terms of the gradient $\nabla^{\Sigma}z$ of the height function $z$. The gradient of $z$ in $\hr$ is simply $\partial_z$, so $\nabla^{\Sigma}z$ is the projection of $\partial_z$ onto $T\Sigma$. In terms of the normal $\nu$ to the surface $\nabla^{\Sigma}z = \partial_z - \la \partial_z, \nu \ra \nu$. From this and Lemma \ref{seccurv} it follows that
\[ K_{\hr}(s,\theta) = -\la \partial_z, \nu \ra^2 = |\nabla^{\Sigma}z|^2 - 1. \]
Now a straightforward, intrinsic calculation of $|\nabla^{\Sigma}z|^2$ using the parametrisation \eqref{parm} and metric \eqref{stmetric} gives 
\[ |\nabla^{\Sigma}z|^2 = \cos^2 \theta \tanh^2s + \sin^2 \theta. \]
Together the last two equations give the result. \qedhere
\end{proof}

By direct calculation from the metric \eqref{stmetric}, or changing variables in Proposition $5.7$ of \cite{DH}, we have a formula for the intrinsic curvature on the horizontal catenoids. 

\begin{Lem}\label{intrcurv}
The intrinsic curvature $K_{\Sigma}$ on the horizontal catenoids is
\[ K_{\Sigma}(s,\theta) = -\frac{\alpha^2(\alpha - \varphi')^2}{\cosh^2s}\left(\frac{1}{\cosh^2s} + \alpha^{-2}\left(\frac{\sin^2 \theta \cos^2\theta}{(\alpha - \varphi')^2} - \frac{\varphi'}{\alpha - \varphi'}(2 \cos^2 \theta - 1) \right) \right). \]
\end{Lem}

Note that the integral of the absolute value of $K_{\Sigma}$ is infinite. 

\begin{Lem}\label{principalcurvest}
The principal curvatures $\kappa_1, \kappa_2$ on $\Sigma_{\epsilon}$ satisfy the estimates
\[ \kappa_1(s,\theta) = 1 + \mathcal{O}(\epsilon^{-2}\cosh^{-4}s + \epsilon^{-1}\cosh^{-2}s), \hspace{5pt} \kappa_2(s,\theta) = \mathcal{O}(\epsilon^{-2}\cosh^{-4}s + \epsilon^{-1}\cosh^{-2}s). \]
\end{Lem}

\begin{proof}
The constant mean curvature condition $\kappa_1 + \kappa_2 = 1$ and the Gauss equation $K_{\hr} = K_{\Sigma} - \kappa_1\kappa_2$ imply that 
\[ \kappa_1^2 + \kappa_2^2 = 1 - 2\kappa_1\kappa_2 = 1 + 2(K_{\hr} - K_{\Sigma}). \] Using the cmc equation once more we see that the principal curvatures satisfy the quadratic equation $\kappa^2 - \kappa + K_{\Sigma} - K_{\hr} = 0$,  whose roots are
\[ \kappa = \frac{1}{2}\left(1 \pm \sqrt{1 - 4(K_{\Sigma} - K_{\hr})}\right). \]  By Lemmas \ref{seccurvcoord} and \ref{intrcurv} we have the bound $K_{\Sigma} - K_{\hr} = \mathcal{O}(\epsilon^{-2}\cosh^{-4}s + \epsilon^{-1}\cosh^{-2}s)$ from which the estimates follow.
 \qedhere
\end{proof}

\subsubsection{Convergence to horocylinders in the small necksize limit}

In this section we represent an annular domain in the horizontal catenoids as a horizontal graph, that is to say a graph of the form $y = g_{\e}(x,z)$. We obtain an expansion of the graphing function $g_{\e}$ which implies that the end of horizontal catenoids converges uniformly to a horocylinder on compact subsets of $\hr - \{(0,0,0)\}$. It will be most convenient to state the results in terms of the following polar coordinates. 

\begin{Def}[Polar coordinates $r,\gamma$]
The polar coordinates $r,\gamma$ in the $x,z$ plane are defined by
\[ x = r\sin \gamma, \hspace{5pt} z = r\cos \gamma. \]
\end{Def}

\begin{Lem}\label{horgraphexplem}
Fix a parameter $\rho > 1$. Then for all small $\epsilon$ (depending on $\rho$) the horizontal graph $g_{\epsilon}$ parametrising $\Sigma_{\epsilon} \cap \{(x,y,z) \in \hr: y \geq 1, \rho^{-1} \leq r \leq \rho \}$ admits the expansion
\begin{equation}
 g_{\epsilon}(r, \gamma) = 1 - \epsilon \log \epsilon + \epsilon \log(2r) + \mathcal{O}_{\rho}(\epsilon^2(\log \epsilon)^2) \hspace{5pt} \mbox{ as } \e \rightarrow 0.
\label{horgraphexp}
\end{equation}
Note that the dominant terms are independent of $\gamma$. 
\end{Lem}

\begin{proof}
We begin by defining an auxiliary variable $\tilde{r}$ which is approximately the radial variable $r$ defined above:
\[ \tilde{r} := \epsilon \cosh s. \]  For each fixed $\epsilon > 0$ this defines a diffeomorphism $(0,\infty) \ni s \leftrightarrow \tilde{r} \in (\epsilon, \infty)$, which we restrict to $\tilde{r} \in [ \rho^{-1}, \rho]$. Inverting this we obtain
\begin{equation} s = \log \left(\epsilon^{-1}\tilde{r} + \sqrt{\epsilon^{-2}\tilde{r}^2 - 1 } \right) = \log (2\epsilon^{-1}\tilde{r}) + \mathcal{O}_{\rho}(\epsilon^2). \label{sitotilder} \end{equation}
Now we derive an expansion for $y_{\e}(s,\theta) = e^{\epsilon s}\omega(s,\theta)$ in terms of $\tilde{r}$. Firstly note that $\omega = 1 + \mathcal{O}(\e^2)$ uniformly on $s \geq 0$. Hence the expansion \eqref{sitotilder} implies
\begin{equation} y_{\e}(s(\tilde{r}),\theta) = 1 - \epsilon \log \epsilon + \epsilon \log (2\tilde{r}) + \mathcal{O}_{\rho}(\epsilon^2(\log \epsilon)^2). \label{yitotilder} \end{equation}
Note that this is essentially \eqref{horgraphexp}, except we want it in terms of $r$ not $\tilde{r}$.  

We must find an expression for $\tilde{r}$ in terms of $r$. To begin with we compute
\[ x_{\e}(s,\theta) = \frac{\sin \theta}{\alpha_*(\alpha - \varphi')}(\cosh s)y_{\e}(s,\theta) = \tilde{r} \sin \theta (1 +  \mathcal{O}_{\rho}(\epsilon|\log \epsilon|), \]
thanks to $\alpha_*^{-1}(\alpha - \varphi')^{-1} = \e(1 + \mathcal{O}(\e))$ and \eqref{yitotilder}. Similarly one has
\[ z_{\e}(s,\theta) = \frac{\cos \theta}{\alpha(\alpha - \varphi')}\cosh s = \tilde{r} \cos \theta (1 +  \mathcal{O}(\epsilon)). \]
From these and the relation $r^2 = x^2 + z^2$ it follows that 
\[ \tilde{r} = r(1 +  \mathcal{O}_{\rho}(\epsilon|\log \epsilon|)). \]
Putting this back into \eqref{yitotilder} completes the proof. \qedhere
\end{proof} 

\subsubsection{Truncating the horizontal catenoids}\label{truncation}

\begin{Def} We will truncate the horizontal catenoids at $s =S_{\e}$, where $S_{\e}$ is defined as the positive solution to 
\[ \e \cosh s = 1. \]
\end{Def}

We think of the surface $\mathbf{X}_{\alpha}\left([S_{\e},\infty)\times \mathbb{S}^1\right)$ as the end of a horizontal catenoid. When we construct cmc $1/2$ surfaces asymptotic to a horizontal catenoid end, it will be as normal graphs over these surfaces.  

Note that 
\begin{equation} S_{\e} = |\log \e| + \mathcal{O}(1) \mbox{ as $\e \rightarrow 0$. } \label{SepExp} \end{equation} Although this choice of cut-off appears to truncate the ends far from the necks, an intrinsic calculation using the metric \eqref{stmetric} shows that the distance from the neck geodesic $s = 0$ to the  boundary curve $s = S_{\e}$ is realised by the geodesic $\theta = 0$, and equals $1 + \mathcal{O}(\e^2)$.

By Lemma \ref{principalcurvest}, the principal curvatures are bounded uniformly in $s$ and $\e$ on the end. Therefore this choice of truncation avoids the neck regions of the catenoids, where the curvatures blow up as $\e \rightarrow 0$. Furthermore, it follows from Lemma \ref{horgraphexplem} that as $\e \rightarrow 0$, these ends converge to the horocylinder $y = 1$, the convergence being uniform on compact subsets of $\hr - \{(0,0,0)\}$. 

For later reference we record here the behaviour of the boundary curve $\mathbf{X}_{\alpha}(S_{\e}, \theta)$ as $\e \rightarrow 0$ when regarded as a horizontal graph over a curve in the $xz$ plane. 

\begin{Lem}\label{BdCurveNecksizeLim}         
On the boundary curve $s = S_{\e}$  of an end we have
\begin{eqnarray}
x_{\alpha}(S_{\e},\theta) & = & \sin \theta + \mathcal{O}(\e|\log\e|), \\
y_{\alpha}(S_{\e},\theta) & = & 1 - \e \log(\e) + \mathcal{O}(\e),  \label{bdcurvegraph}\\
z_{\alpha}(S_{\e},\theta) & = & \cos \theta + \mathcal{O}(\e).
\end{eqnarray} 
Thus the boundary curve is the graph over a curve in the $xz$ plane that converges to the unit circle $x^2 + z^2 = 1$, and the graph admits the expansion \eqref{bdcurvegraph}, as $\e \rightarrow 0$. 
\end{Lem}

\begin{proof}
The proof is by straightforward calculation, using \eqref{SepExp}. \qedhere
\end{proof}

\newpage

\section{Expansion of the Mean Curvature Operator}

The goal of this section is to establish a first-order expansion for the mean curvature operator for normal graphs off the ends of horizontal catenoids. It's a well-known fact (\cite{MP}, \cite{MPP}) that the linearisation of the mean curvature operator for a surface $\Sigma$ is the Jacobi operator $\mathcal{J}$, given by
\[ \mathcal{J} = \triangle + |h|^2 + Ric(\nu). \]
Here $\triangle$ is the Laplace-Beltrami operator on $\Sigma$, $|h|^2$ is the norm squared of the second fundamental form, and $Ric(\nu)$ is the Ricci tensor evaluated on the unit normal field $\nu$ of $\Sigma$. The emphasis in the following result is therefore on the bounds for the coefficients of the non-linear error $\mathcal{Q}_{\e}$, or more precisely the estimate \eqref{QLip} that we will use in the contraction mapping argument. 

\begin{Prop}\label{MeanCurvExp}
The mean curvature operator $\mathcal{M}_{\e}(w)$ admits the expansion
\[ \mathcal{M}_{\e}(w) = 1/2 + \mathcal{J}_{\e}w + \mathcal{Q}_{\e}(w). \]
The error term $\mathcal{Q}_{\e}$ has the properties that $\mathcal{Q}_{\e}(0) = 0$, and there exists a constant $C > 0$, independent of $\epsilon$, such that for all $s \geq S_{\epsilon}$
\begin{equation} |\mathcal{Q}_{\e}(v) - \mathcal{Q}_{\e}(w)|_{0,\alpha,s} \leq C\left(|v|_{2,\alpha,s} + |w|_{2,\alpha,s} \right)|v - w|_{2,\alpha,s}. \label{QLip} \end{equation} 
\end{Prop}

To prove Proposition \ref{MeanCurvExp} we work partly abstractly, using Fermi coordinates on a tubular neighbourhood of an embedded surface. This simplifies the calculation of the zeroth and first order terms in the expansion. To handle the higher order terms, and make estimates that are valid uniformly on the ends, it is necessary to introduce various facts about the geometry of $\hr$ and the ends of horizontal catenoids.

\subsection{Fermi coordinates}

Consider an embedding $\Psi$ of a two-sided surface $\Sigma$ in a three-manifold $(\tM, \tg)$, with unit normal $\nu$. Relative to any (local) coordinates $x = (x^1,x^2)$ on $\Sigma$, one defines the \emph{Fermi coordinates} on a tubular neighbourhood of (a coordinate neighbourhood in) $\Sigma$, by the parametrisation
\[ (x^1,x^2,t) \mapsto exp_{\Psi(x)}(t\nu(x)). \]
Here $exp_{\Psi(x)}$ denotes the exponential map on $T_{\Psi(x)}\tM$. 

Given a sufficiently small and smooth function $w$ on $\Sigma$, we define the surface $\Sigma_w$ by the embedding
\[ \Sigma \ni x \mapsto \Psi_w(x) := exp_{\Psi(x)}(w(x)\nu(x)) \in \tM. \] With respect to Fermi coordinates, the embedding $\Psi_w$ takes the simple form
\[ (x^1,x^2) \mapsto (x^1,x^2, w(x^1,x^2)). \]
This greatly facilitates the calculations, which proceed analogously to those for the graph of a function in Euclidean space. Before we compute the induced metric $g_w$ and second fundamental form $h_w$ on $\Sigma_w$, we develop a few aspects of the theory of Fermi coordinates
that we will need.

\subsubsection*{Notation}
The coordinate frame in $\tM$ determined by the Fermi coordinates $(x^1,x^2,t)$ will be denoted by $\partial_1,\partial_2$ and $\partial_3$. The level sets of the coordinate $t$ will be denoted by $\Sigma_t$, and refered to as \emph{tubular hypersurfaces}.

\begin{Lem}[Tubular Gauss Lemma]
The $t$-coordinate lines are unit speed geodesics that are orthogonal to the tubular hypersurfaces. 
\end{Lem}

\begin{proof}
By definition of the exponential map, the $t$-coordinate lines are unit speed geodesics. Note that $\partial_1$ and $\partial_2$ are tangential to the tubular hypersurfaces for all $t$, so we need to show that $\partial_3$ is orthogonal to them. Since $\nu(x) = \partial_3|_{(x,0)}$, we have $\la \partial_i, \partial_3 \ra = 0$ when $t = 0$. Furthermore, for $i = 1,2$,
\[ \frac{d}{dt}\la \partial_i, \partial_3 \ra = \la \tN_{\partial_3} \partial_i, \partial_3 \ra + \la \partial_i, \tN_{\partial_3}\partial_3 \ra = 0, \]
because each term vanishes, since $t$-coordinate lines are unit speed geodesics.  \qedhere
\end{proof}

In the forthcoming calculations we find it convenient, following \cite{G}, to use the following two fields of operators, which satisfy a differential equation.

\begin{Def}
The operators $S$ and $R$ are defined by
\[ S(\cdot) = -\tN_{(\cdot)}\partial_3, \hspace{5pt} \mbox{ and } \hspace{5pt} R(\cdot) = \widetilde{R}(\cdot, \partial_3)\partial_3, \]
where $\widetilde{R}(X,Y)Z$ is the curvature endomorphism of $\tM$. 
\end{Def}

By the Tubular Gauss Lemma, for each fixed $t$ the operator $S$ is nothing but the shape operator for the tubular hypersurface $\Sigma_t$. The relevance of $R$ is a little less clear, but illuminated by the next two lemmas. 

\begin{Lem}
\label{Rlem} For each fixed $t$ the operators $R$ are a field of symmetric linear operators on $T\Sigma_t$ whose eigenvalues are sectional curvatures (in $\tM$) of two-planes orthogonal to the tubular hypersurfaces.
\end{Lem}

\begin{proof}
Note that
\begin{equation} \la R(X), Y \ra = \la \widetilde{R}(X,\partial_3)\partial_3, Y \ra = \widetilde{Rm}(X,\partial_3, \partial_3, Y). \label{REq} \end{equation}
Taking $Y = \partial_3$, the anti-symmetry of $\widetilde{Rm}$ in its last two arguments shows that $R(X)$ is tangential to the tubular hypersurfaces (by the Tubular Gauss lemma), so the operators $R$ reduce to operators on $T\Sigma_t$. Symmetry of $R$ also follows from the symmetries of the curvature tensor $\widetilde{Rm}$. Hence the eigenvalues of $R$ are the extrema of the associated quadratic form $\la R(X),X \ra$, when restricted to the unit sphere in the tangent spaces for $\Sigma_t$. By \eqref{REq} these extrema are the sectional curvatures of certain two-planes orthogonal to the tubular hypersurfaces. \qedhere
\end{proof}

The following differential equation for the shape operators $S$ is the key fact about $S$ and $R$. Recall that the connection on $Hom(T\tM)$ is defined to respect the pairing of an operator with a tangent vector. Hence for an operator $T$ we have
\[ (\tN_XT)Y := \tN_X(TY) - T\left(\tN_XY\right). \]

\begin{Lem}[Ricatti equation for shape operators (\cite{G})]
\label{RicattiLem}
The shape operators $S$ satisfy the Ricatti equation
\begin{equation} \nabla_{\partial_3}S = S^2 + R. \label{RicattiS} \end{equation}
\end{Lem}

\begin{proof}
Equation \eqref{RicattiS} is tensorial, so we can verify it using any basis we please. Using coordinate vectors $\partial_i$, one has simply $R \partial_i = - \tN_{\partial_3}\tN_{\partial_i}\partial_3$, therefore
\[ \left(\tN_{\partial_3} S\right) \partial_i = \tN_{\partial_3}(S\partial_i) - S\left(\tN_{\partial_3}\partial_i\right) =
 - \tN_{\partial_3} \tN_{\partial_i} \partial_3 - S \left(\tN_{\partial_i}\partial_3\right) = R \partial_i + S(S(\partial_i)). \qedhere \]
\end{proof}

There is also a scalar version of this result for the eigenvalues of the operators $S$, which will be an important ingredient in our estimates later.

\begin{Cor}\label{eigenvalRicatti}
Let $\kappa_1, \kappa_2$ be the eigenvalues of $S$ (which for fixed $t$ are also the principal curvatures of the tubular hypersurfaces). We assume that the eigenvalues are distinct. Then for $i = 1,2$ they satisfy the equation
\[ \frac{\partial \kappa_i}{\partial t} = \kappa_i^2 + \rho_i, \]
where $\rho_i$ is the sectional curvature of  the plane spanned by $\partial_3$ and the  $\kappa_i$-eigenspace of $S$. 
\end{Cor}

\begin{proof}
$S$ is symmetric, therefore on each tubular hypersurface there is an orthonormal frame of eigenvectors for $S$ which we denote $\xi_i$, $i = 1,2$. Applying the equation \eqref{RicattiS} to $\xi_i$ we get
\[ \frac{\partial \kappa_i}{\partial t}\xi_i + \kappa_i \tN_{\partial_3}\xi_i - S\left( \tN_{\partial_3}\xi_i\right) = \kappa_i^2\xi_i +  \widetilde{R}(\xi_i,\partial_3)\partial_3. \]
Taking the inner product of this with $\xi_i$ gives the result, because symmetry of $S$ implies that
\[ \kappa_i \la  \tN_{\partial_3}\xi_i, \xi_i \ra =  \la S\left( \tN_{\partial_3}\xi_i\right), \xi_i \ra. \qedhere \] 
\end{proof}

Now we are ready to prove the main result we will use about Fermi coordinates, which is the expansion of part of the ambient metric and Christoffel symbols in the $t$ direction.

\begin{Prop}\label{ambientexpansion}
We work in Fermi coordinates. For $i,j \in \{1,2\}$ the  metric components $\tg_{ij}$ and Christoffel symbols $\tG_{ij}^3$ have the expansions
\begin{eqnarray*}
\tg_{ij}(x,t) & = & g_{ij}(x) - 2h_{ij}(x)t + \tQ_{ij}(x,t) \\
\tG_{ij}^3(x,t) & = & h_{ij}(x) - p_{ij}(x)t + \tQ'_{ij}(x,t),  
\end{eqnarray*}
where $g$ and $h$ are the metric and second fundamental form on $\Sigma$, and $p$ is a symmetric two-tensor given by $p_{ij} = \la (S^2 - R)\partial_i, \partial_j \ra$. Furthermore $\tQ_{ij}$ and $\tQ_{ij}'$ satisfy the Lipschitz estimates 
\begin{eqnarray*}
|\tQ_{ij}(x,t_1) - \tQ_{ij}(x,t_2)| & \leq & \sup_{|t| \leq 1/4}\left|\frac{\partial^2\tg_{ij}}{\partial t^2}(x,t) \right|\left(|t_1| + |t_2|\right) |t_1 - t_2| \\
|\tQ'_{ij}(x,t_1) - \tQ'_{ij}(x,t_2)| & \leq &  \sup_{|t| \leq 1/4}\left|\frac{\partial^3\tg_{ij}}{\partial t^3}(x,t) \right| \left(|t_1| + |t_2|\right) |t_1 - t_2|,
\end{eqnarray*}
for all $|t_1|,|t_2| \leq 1/4$. 
\end{Prop}

\begin{proof}
We begin by computing the first three partial derivatives of $\tg_{ij}(x,t)$ with respect to $t$. In these calculations we repeatedly use standard facts about the Levi-Civita connection, such as its symmetry and compatibility with the metric, as well as the symmetry of $S$. Thus
\begin{equation} \frac{\partial \tg_{ij}}{\partial t}(x,t) = - 2\la S \partial_i, \partial_j \ra. \label{FirstDerivative} \end{equation}
Then since $R\partial_i = - \tN_{\partial_3}\tN_{\partial_i}\partial_3$ we find
\begin{equation} \frac{\partial^2 \tg_{ij}}{\partial t^2}(x,t) = 2 \la (S^2 - R)\partial_i, \partial_j \ra. \label{SecondDerivative} \end{equation}
For the third derivative we begin with
\[ \frac{\partial^3 \tg_{ij}}{\partial t^3}(x,t) = 2\la \tN_{\partial_3}\left((S^2 - R)\partial_i\right), \partial_j \ra - 2\la S(S^2 - R)\partial_i, \partial_j \ra. \]
Expanding 
\[ \tN_{\partial_3}\left((S^2 - R)\partial_i\right) = \left(\tN_{\partial_3}(S^2 - R)\right)\partial_i - (S^2 - R)(S\partial_i), \]
and using the Ricatti equation to compute $\tN_{\partial_3}S^2= 2S^2 + RS + SR$ gives
\begin{equation} \frac{\partial^3 \tg_{ij}}{\partial t^3}(x,t) = 2\la \left(2(RS + SR) - \tN_{\partial_3}R\right)\partial_i, \partial_j \ra. \label{ThirdDerivative} \end{equation}
Now any twice differentiable function $\phi$ on $(-1/4,1/4)$ has an expansion
\[ \phi(t) = \phi(0) + \phi'(0)t + q(t), \]
where, by the mean value inequality, the function $q(t)$ satisfies the estimate
\[ |q(t_1) - q(t_2)| \leq \sup_{|t| \leq 1/4}\left|\phi''(t) \right|\left(|t_1| + |t_2|\right) |t_1 - t_2|. \]
The expansion for $\tg_{ij}$ follows from this and \eqref{FirstDerivative}. Finally the standard formula for Christoffel symbols in terms of the inverse and first derivatives of the metric gives us
\[ \tG_{ij}^3 = -\frac{1}{2} \frac{\partial \tg_{ij}}{\partial t}. \]
Hence we obtain the expansion for $\tG_{ij}^3$ from the formulae \eqref{FirstDerivative}, \eqref{SecondDerivative}. 
\qedhere
\end{proof}

\subsubsection*{A preliminary expansion of $g_w$ and $h_w$.}

Now we return to the surface $\Sigma_w$ for which we require expansions of the metric $g_w$ and second fundamental form $h_w$. Recall that in Fermi coordinates, a parametrisation of this surface is given by
\[ (x^1, x^2) \mapsto (x^1,x^2, w(x^1,x^2)). \]
Therefore the tangent vectors are $\partial_1 + w_1\partial_3$ and $\partial_2 + w_2\partial_3$, where we use the notation $w_i = \frac{\partial w}{\partial x^i}$ for brevity. A succinct expression for the unit normal $\nu_w$ to $\Sigma_w$ can be given by introducing the vector
\[ Dw := (\Psi_w)_*(\nabla_{\Sigma_w} w) = (g_w)^{ij}w_i(\partial_j + w_j\partial_3) \]
that is, the pushforward of the gradient $\nabla_{\Sigma_w} w$ of $w$ on $\Sigma_w$ by the embedding $\Psi_w$. Then one can easily verify that 
\[ \nu_w = \frac{\partial_3 - Dw}{\sqrt{1 - (g_w)^{ij}w_iw_j}}. \]

A straighforward calculation using the Tubular Gauss Lemma proves
\begin{Lem}\label{PrelimForms}
The metric and second fundamental of $\Sigma_w$ are
\begin{eqnarray*}
(g_w)_{ij}(x) & = & \tg_{ij}(x,w) + w_iw_j \\
\sqrt{1 - (g_w)^{ij}w_iw_j} \cdot (h_w)_{ij}(x) & = & \tG^3_{ij}(x,w) + w_{ij} - (g_w)^{kl}w_k\la \tN_{\partial_i}\partial_j, \partial_l \ra \\
&& +  (g_w)^{kl}w_k\left\{ w_i \tG_{jl}^3 + w_j\tG_{il}^3 - w_l\tG_{ij}^3 - w_lw_{ij}\right\}. 
\end{eqnarray*}
\end{Lem}


For the moment we will pause in the calculation of the mean curvature, and collect certain estimates relating to the geometry of $\hr$, the horizontal catenoids, and Fermi coordinates defined relative to the $(s,\theta)$ coordinates on the horizontal catenoids. This allows us to give a simpler treatment of the otherwise complicated quantities that will contribute to the non-linear error in the expansion of the mean curvature. 

\subsection{Estimates for error terms}

In computing an expansion for the mean curvature operator we'll need control over the errors $\tQ_{ij}$ and $\tQ'_{ij}$ - and also the main terms - in the expansions in Proposition \ref{ambientexpansion}. More precisely, we need estimates on these terms that hold both uniformly on the ends of horizontal catenoids, and uniformly for all small $t$. Our strategy to obtain these begins with the observation that all of these terms are the components of tensors defined in terms of the operators $R$ and $S$. Therefore, if one can estimate the eigenvalues of these operators, one will have estimates for these terms.

\begin{Lem}\label{EvalBound}
The eigenvalues of $R$ are bounded in $[-1,0]$. The eigenvalues of $S$ are bounded on a uniform tubular neighbourhood of the ends, and uniformly for all small $\e$. More precisely, there exists a constant $C > 0$ such that for all small $\e > 0$, we have
\[ |\kappa_i(s,\theta,t)| \leq C \hspace{5pt} \mbox{ for all } (s,\theta,t) \in [S_{\e},\infty)\times \mathbb{S}^1\times [-1/4,1/4], i = 1,2. \]
\end{Lem}

\begin{proof}
In Lemma \ref{Rlem} we showed that the eigenvalues of $R$ are sectional curvatures in $\hr$, hence the first claim. For the second, note that Lemma \ref{principalcurvest} implies the estimates $\kappa_1(s,\theta,0) = 1 + \mathcal{O}(\epsilon)$ and $\kappa_2(s,\theta,0) = \mathcal{O}(\epsilon)$, uniformly on the ends. To leverage these from $t = 0$ to estimates for $t \in [-1/4, 1/4]$ we use the Ricatti equation satisfied by the eigenvalues (Corollary \ref{eigenvalRicatti}). 

For brevity of notation write $\kappa_i(t)$ for $\kappa_i(s,\theta, t)$. Now the functions $\rho_i$ in the Ricatti equation
\[ \frac{\partial \kappa_i}{\partial t} = \kappa_i^2 + \rho_i \]
satisfy the bounds $-1 \leq \rho_i \leq 0$, because they are sectional curvatures in $\hr$. Hence we define $\ok_i$ and $\uk_i$ as solutions to the initial value problems
\[ \frac{\partial \ok_i}{\partial t} = \ok_i^2, \hspace{5pt} \ok_i(0) = \kappa_i(0); \hspace{10pt} \frac{\partial \uk_i}{\partial t} = \uk_i^2 - 1, \hspace{5pt} \uk_i(0) = \kappa_i(0). \]
By comparison of the equations they satisfy, $\uk_i(t) \leq \kappa_i(t) \leq \ok_i(t)$ for all small $t \geq 0$, and $\ok_i(t) \leq \kappa_i(t) \leq \uk_i(t)$ for all small $t \leq 0$. Hence
\[ |\kappa_i(t)| \leq \max \{ |\ok_i(t)|, |\uk_i(t)| \} \mbox{ for all small $t$, }i = 1,2. \]
Integrating the equations we find
\[ \ok_i(t) = \frac{\kappa_i(0)}{1 - \kappa_i(0)t}, \hspace{10pt} \uk_i(t) = \frac{\kappa_i(0) + 1 + (\kappa_i(0) - 1)e^{2t}}{\kappa_i(0) + 1 - (\kappa_i(0) - 1)e^{2t}}. \]
Then since $\kappa_1(0) = 1 + \mathcal{O}(\e)$ and $\kappa_2(0) = \mathcal{O}(\e)$, 
\[ \ok_1(t) = \frac{1 + \mathcal{O}(\e)}{1 - (1 + \mathcal{O}(\e))t}, \hspace{5pt} \ok_2(t) = \frac{\mathcal{O}(\e)}{1 - \mathcal{O}(\e)t}, \hspace{5pt} \uk_1(t) = \frac{2 + \mathcal{O}(\e)(1 + e^{2t})}{2 + \mathcal{O}(\e)(1 - e^{2t})}, \mbox{ and } \]
\[ \uk_2(t) = \frac{1 - e^{2t} + \mathcal{O}(\e)(1 + e^{2t})}{1 + e^{2t} + \mathcal{O}(\e)(1 - e^{2t})}. \]
Clearly these are all bounded when $\epsilon$ is sufficiently small and $|t| \leq 1/4$, and the result follows. \qedhere
\end{proof}

\begin{Lem}\label{nablaRvanish}
Let $\sigma_1(t) \geq \sigma_2(t)$ be the eigenvalues of $R(t)$. Then $\sigma_1(t) \equiv 0$, and
\[ \sigma_2(t) = \left \langle \partial_3, \xi \right \rangle^2 - 1. \]
Hence $\sigma_1$ and $\sigma_2$ are independent of $t$, and $\nabla_{\partial_3}R$ vanishes identically. 
\end{Lem}

\begin{proof}
Recall that the eigenvalues of $R(t)$ are sectional curvatures of two-planes spanned by $\partial_3$ and vectors tangent to the tubular hypersurfaces. Now the first case to consider is that $\partial_3$ is vertical, in which case the $t$-geodesic is a vertical geodesic and $\partial_3$ remains vertical for all $t$. Therefore any two plane spanned by $\partial_3$ is vertical, all of their sectional curvatures are zero, and $R(t)$ vanishes identically for all $t$. In this case all the assertions of the lemma certainly hold. 

In the second case we assume that $\partial_3$ is not vertical. Since $\partial_z$ is parallel we have
\[ \partial_t \left \langle \partial_3, \partial_z \right \rangle = \left \langle \nabla_{\partial_3}\partial_3, \partial_z \right \rangle + \left \langle \partial_3, \nabla_{\partial_3}\partial_z \right \rangle = 0, \] and so the angle between $\partial_3$ and $\partial_z$ is independent of $t$. Now we know that the sectional curvatures of $\mathbb{H}^2 \times \mathbb{R}$ are bounded between $-1$ and $0$, and it is clear that the $2$-plane spanned by $\partial_3$ and the tangential projection $\partial_z^T$ of $\partial_z$ onto the tangent space of the tubular hypersurface contains $\partial_z$. Therefore the sectional curvature of this plane is zero, the quadratic form associated to $R(t)$ is maximised in the direction of $\partial_z^T$, and so $\partial_z^T$ must be an eigenvector of $R(t)$ with eigenvalue $0$. 

To compute the second eigenvalue we observe that, since $R(t)$ is symmetric, the eigenspace for $\sigma_2(t)$ is orthogonal to $\partial_z^T$, and tangential to the tubular hypersurface. Hence the plane spanned by this eigenspace and $\partial_3$ is orthogonal to $\partial_z^T$, and $\sigma_2(t)$ is the sectional curvature of this plane. But Lemma \ref{seccurv} tells us that we can compute this sectional curvature as
\[ -\left( \left \langle \frac{\partial_z^T}{|\partial_z^T|}, \partial_z \right\rangle \right)^2 = - |\partial_z^T|^2 = \left \langle \partial_3, \partial_z \right \rangle^2 - 1. \]This establishes the formula for $\sigma_2(t)$. Since it only depends on $\left \langle \partial_3, \partial_z \right \rangle$ and since this is independent of $t$ as we have mentioned, this completes the proof of the assertions regarding the eigenvalues. 

To prove the vanishing of $\nabla_{\partial_3}R$ we begin by completing the orthonormal set of vector fields $\partial_3$ and $\hat{\partial_z}^T := \frac{\partial_z^T}{|\partial_z^T|}$ to a smooth orthonormal frame along the $t$-geodesic with a vector field $\zeta(t)$ along the curve. As already noted, $\zeta(t)$ is an eigenvector for $\sigma_2(t)$. Moreover $\zeta(t)$ is actually parallel along this geodesic, because $\partial_3$ and $\frac{\partial_z^T}{|\partial_z^T|}$ are. Now expand any vector field along the geodesic that is tangential to tubular hypersurfaces as $X(t) = a(t)\hat{\partial_z}^T + b(t)\zeta$. Using the facts just mentioned, and that $\sigma_1$ and $\sigma_2$ are independent of $t$, we compute
\begin{eqnarray*} (\nabla_{\partial_3}R)X & = & \nabla_{\partial_3}(RX) - R(\nabla_{\partial_3}X) \\
& = & \nabla_{\partial_3}(b(t)\sigma_2 \zeta) - R(a'(t)\hat{\partial_z}^T + b'(t)\zeta) = b'(t)\sigma_2 \zeta - b'(t)\sigma_2 \zeta = 0. \qedhere \end{eqnarray*}
\end{proof} 


\begin{Cor}\label{BoundsForForms}
There exists $t_0 > 0$ and a constant $C > 0$ independent of $\epsilon$, such that
\[ \sup_{\theta \in \mathbb{S}^1, |t| \leq t_0. k \in \{0,1,2,3\}} \left| \frac{\partial^k\tg_{ij}}{\partial t^k} (s,\theta,t)\right| \leq C\epsilon^2 \cosh^2s \hspace{5pt} \mbox{ for all } s \geq S_{\e}. \]
Note that this implies an estimate of the same form for $\tsff_{ij}$ and $\tp_{ij}$ too, by \eqref{FirstDerivative} and \eqref{SecondDerivative}. 
\end{Cor}

\begin{proof}
We claim that once an estimate of this form is proved for $\tg_{ij}$, those for the derivatives follow. Indeed the formulae \eqref{FirstDerivative}, \eqref{SecondDerivative} and \eqref{ThirdDerivative} for the first three derivatives of $\tg_{ij}$ with respect to $t$ show that we can estimate them in terms of $\tg_{ij}$ and the operator norms of $R,S$ and $\nabla_{\partial_3}R$, by Cauchy-Schwarz' inequality. Moreover the operator norms of $R$ and $S$ are bounded uniformly in all the variables, by Lemma \ref{EvalBound}, and $\nabla_{\partial_3}R = 0$ by Lemma \ref{nablaRvanish}. This proves the claim. 

To obtain the estimate for $\tg_{ij}$, first note that
\[ \tg_{ij}(s,\theta,t) = g_{ij}(s,\theta) + \tL_{ij}(s,\theta,t), \] where, by the mean value inequality,
\[ |\tL_{ij}(s,\theta,t)| \leq \sup_{|\tau| \leq |t|}\left| \frac{\partial \tg_{ij}}{\partial t}(s,\theta, \tau)\right| \cdot |t|. \]
Since $\frac{\partial \tg_{ij}}{\partial t} = 2\la S\partial_i, \partial_j \ra$, Lemma \ref{EvalBound} gives
\begin{equation} |\tL_{ij}(s,\theta,t)| \leq C\sqrt{ \tg_{ii}(s,\theta,t) \tg_{jj}(s,\theta,t)}\cdot |t|. \label{LDiag} \end{equation}
When $i = j$, this yields $\tg_{ii}(s,\theta, t) = g_{ii}(s,\theta) +  \mathcal{O}(\tg_{ii}(s,\theta,t))\cdot|t|$, or in other words,
\[ (1 + \mathcal{O}(|t|))\tg_{ii}(s,\theta, t) = g_{ii}(s,\theta). \]
Since $g_{ij}(s,\theta) = \mathcal{O}(\epsilon^2\cosh^2s)$, we obtain the desired estimate in the case $i = j$. The general case then follows from the estimate when $i = j$ and \eqref{LDiag}. \qedhere
\end{proof}

\subsection{Proof of Proposition \ref{MeanCurvExp}}
For notational simplicity, we will use the symbols $Q_{-1}, Q_0$ and $Q_1$ to denote ''quadratic'' errors and the growth rates of their coefficients, as follows. 

\begin{Def}For $k \in \{-1,0,1\}$, a quantity will be denoted $Q_k(w)$ if $Q_k(0) = 0$, and there is a constant $C > 0$, independent of $\e$, such that for all $s \geq S_{\e}$ we have
\[ |Q_k(v) - Q_k(w)|_{0,\alpha,s} \leq C(\e^2 \cosh^2 s)^k(|v|_{2,\alpha,s} + |w|_{2,\alpha,s}
 )|v - w|_{2,\alpha,s}. \]
 Note that every error of the form $Q_k$ is also of the form $Q_l$ if $k \leq l$.  
\end{Def}

\begin{Lem}
The first and second fundamental form of $\Sigma_w$ are
\[ (g_w)_{ij} = g_{ij} - 2h_{ij}w + Q_1(w) \hspace{5pt} \mbox{ and } 
(h_w)_{ij}= h_{ij} - p_{ij}w + (\nabla_g^2w)_{ij} + Q_1(w), \]
where $\nabla_g^2w$ is the Hessian of $w$ on $\Sigma$. 
\end{Lem}

\begin{proof}
It follows from Lemma \ref{PrelimForms} that $(g_w)_{ij} = \tg_{ij}(x,w) + Q_0(w)$. And combining the estimates of Proposition \ref{ambientexpansion} with the those of Corollary \ref{BoundsForForms} gives 
\[ \tg_{ij}(x,w) = g_{ij}(x) - 2h_{ij}(x)w + Q_1(w). \] Together these prove the expansion for $g_w$.

Going back to the preliminary expression of Lemma \ref{PrelimForms} we found
\begin{eqnarray}
\sqrt{1 - (g_w)^{ij}w_iw_j} \cdot (h_w)_{ij}(x) & = & \tG^3_{ij}(x,w) + w_{ij} - (g_w)^{kl}w_k\la \tN_{\partial_i}\partial_j, \partial_l \ra \label{hwref} \\
&& +  (g_w)^{kl}w_k\left\{ w_i \tG_{jl}^3 + w_j\tG_{il}^3 - w_l\tG_{ij}^3 - w_lw_{ij}\right\}. \nonumber
\end{eqnarray}
Once again combining the estimates of Proposition \ref{ambientexpansion} with the those of Corollary \ref{BoundsForForms} gives
\[ \tG_{ij}^3(x,w) = h_{ij}(x) - p_{ij}(x)w + Q_1(w) \hspace{5pt} \mbox{ and } \hspace{5pt}
(g_w)^{kl}w_k\la \tN_{\partial_i}\partial_j, \partial_l \ra = \Gamma_{ij}^k w_k + Q_0(w). \]
Furthermore $\sqrt{1 - (g_w)^{ij}w_iw_j} = 1 + Q_{-1}(w)$, whilst the expression on the second line in the formula \eqref{hwref} is $Q_0(w)$. Therefore we get
\[ (h_w)_{ij} = h_{ij} - p_{ij}w  + w_{ij} + \Gamma_{ij}^kw_k + Q_1(w). \]
Since $w_{ij} + \Gamma_{ij}^kw_k$ is the $ij^{th}$ component of the Hessian of $w$ this completes the proof. \qedhere
\end{proof}
 
\begin{proof}[Proof of Proposition \ref{MeanCurvExp}]
We work in matrix notation, so
\[ g_w = g - 2hw + Q_1(w), \hspace{5pt} \mbox{ and } \hspace{5pt} h_w = h - pw + \nabla_g^2w + Q_1(w). \]
Inverting the expansion for $g_w$ one obtains
\[ g_w^{-1} = g^{-1} + 2g^{-1}hg^{-1}w + Q_{-1}(w). \]
Hence
\[ g_w^{-1}h_w = g^{-1}h + g^{-1}\nabla_g^2w - g^{-1}pw + 2g^{-1}hg^{-1}hw + Q_0(w). \]
In applying the matrix trace, note that
\[trace(g^{-1}\nabla_g^2w) = \triangle w, \hspace{5pt} trace(g^{-1}p) = trace(S^2 - R) = |h|^2 - Ric(\nu), \]
and $(g^{-1}hg{-1}h)_{ij} = g^{ik}h_{kl}g^{lp}h_{pj}$, so
\[  trace(g^{-1}hg^{-1}h) = g^{ik}h_{kl}g^{lp}h_{pi} = h^{ip}h_{pi} = |h|^2. \]
Therefore, as claimed,
\[ \mathcal{M}(w) = \mathcal{M}(0) + \triangle w + (|h|^2 + Ric(\nu))w + \mathcal{Q}(w) \]

The error is of type $Q_0$ (rather than $Q_1$ or $Q_{-1}$) because each term that contributes to it is a product of a term from the expansion of $g_w^{-1}$, which has coefficients that are $\mathcal{O}(\e^{-2} \cosh^{-2}s)$, and a term from the expansion of $h_w$, which has coefficients that are $\mathcal{O}(\e^2 \cosh^2s)$. \qedhere
\end{proof}

\newpage

\section{Linear Analysis of the Jacobi Operator}

In the previous section we expanded the mean curvature operator on normal graphs as
\[ \mathcal{M}_{\e}(w) = \mathcal{M}_{\e}(0) + \mathcal{J}_{\e}w + \mathcal{Q}_{\e}(w). \]
Of course, $\mathcal{M}_{\e}(0)$ is the mean curvature of the horizontal catenoids, and thus identically $1/2$. Therefore a normal graph has constant mean curvature $1/2$ if and only if the function $w$ satisfies 
\[ \mathcal{J}_{\e}w = - \mathcal{Q}_{\e}(w). \]
To solve this equation we will invert the Jacobi operator $\mathcal{J}_{\e}$, reformulate it as a fixed point problem, and apply the contraction mapping principle. This requires that we choose suitable Banach spaces on which the Jacobi operator may be wholly or partially inverted.

We begin our treatment with an explicit calculation of the Jacobi operator in $(s,\theta)$ coordinates.

\begin{Prop}
\label{JacobiExactProp}
The Jacobi operator $\mathcal{J}_{\epsilon}$ is given by
\begin{equation}
\mathcal{J}_{\epsilon} = \frac{\alpha^2(\alpha - \varphi')^2}{\cosh^2s}\left(\partial_s^2 + \partial_{\theta}^2 + \frac{2}{\cosh^2s} + \alpha^{-2} E \right) 
\label{Jacobiexact}
\end{equation} where $E$ is a second order operator, independent of $s$ and $\e$, given by
\begin{equation}
E := \frac{(\cos(2\theta) + 1)}{2}\partial_{\theta}^2 - \frac{\sin(2\theta)}{2}\partial_{\theta} + \cos(2\theta).
\label{eplicitE}
\end{equation}
\end{Prop}

\begin{proof}
As we have already shown, $\mathcal{J}_{\epsilon} = \triangle + |h|^2 + Ric(\nu)$. From the coordinate expression \eqref{stmetric} for the metric, we calculate the Laplacian to be
\[ \triangle  =   \frac{\alpha^2(\alpha - \varphi')^2}{\cosh^2s}\left(\partial_s^2 + \partial_{\theta}^2 + \frac{(\cos(2\theta) + 1)}{2\alpha^2} \partial_{\theta}^2 - \frac{\sin(2\theta)}{2\alpha^2}\partial_{\theta}\right). 
\]
To compute the potential we use the intrinsic curvature $K_{\Sigma}$ of $\Sigma$, and  the Gauss equation to write
\[ K_{\hr} = K_{\Sigma} - \kappa_1\kappa_2. \]
Recall that $\kappa_1, \kappa_2$ are the principal curvatures, and $K_{\hr}$ is the sectional curvature in $\hr$ of the tangent plane to $\Sigma$.
This and the constant mean curvature condition $\kappa_1 + \kappa_2 \equiv 1$ give
\[ |h|^2 = \kappa_1^2 + \kappa_2^2 = (\kappa_1 + \kappa_2)^2 - 2\kappa_1\kappa_2 = 1 + 2K_{\hr} - 2K_{\Sigma}. \] Since $K_{\hr} + Ric(\nu) = -1$, by Lemma \ref{secplusricci},
\[ |h|^2 + Ric(\nu) = 1 + 2K_{\hr} - 2K_{\Sigma} + Ric(\nu) = K_{\hr} - 2K_{\Sigma}. \] 
Finally, combining the formulae for $K_{\hr}$ and $K_{\Sigma}$ from Lemmas \ref{seccurvcoord} and \ref{intrcurv}, and making essential use of the identity $(\varphi')^2 = \alpha^2 + \cos^2 \theta$, one calculates that
\[ |h|^2 + Ric(\nu) = \frac{\alpha^2(\alpha - \varphi')^2}{\cosh^2s}\left(\frac{2}{\cosh^2s} + \frac{\cos(2\theta)}{\alpha^2}\right). \]
This and the expression for the Laplacian complete the proof. \qedhere
\end{proof} 

For applications to the mean curvature equation it will suffice to study the operator
\begin{equation} \mathcal{L}_{\e} := \partial_s^2 + \partial_{\theta}^2 + \frac{2}{\cosh^2s} + \alpha^{-2}E, 
\label{calLepsilon}
\end{equation}
which we shall also call the Jacobi operator. Notice that $\mathcal{L}_{\e}$ is a perturbation of 
\begin{equation}
\mathcal{L}_0 := \partial_s^2 + \partial_{\theta}^2 + \frac{2}{\cosh^2s}.
\label{calLzero}
\end{equation}
This operator is nothing other than the Jacobi operator on a minimal catenoid in euclidean space with its standard parametrisation. A good deal is known about this operator, and we will exploit these facts to study $\mathcal{L}_{\e}$. To begin with, a useful class of Banach spaces on which to study such operators are the weighted H\"older spaces \cite{PR}. The ones we shall work with are the standard ones on a half-cylinder. 

\begin{Def}
Fix $S > 0$, a non-negative integer $k$, $\alpha \in (0,1)$ and $\mu \in \mathbb{R}$. $\mathcal{C}^{k,\alpha}_{\mu}([S, \infty) \times \mathbb{S}^1)$ is defined as the set of $\mathcal{C}^{k,\alpha}_{loc}([S, \infty) \times \mathbb{S}^1)$ functions for which the following weighted norm is finite
\[ ||u||_{k, \alpha,\mu} := \sup_{s \geq S} (\cosh s)^{-\mu}|u|_{k,\alpha, s}. \]
Here $|u|_{k,\alpha, s}$ is the usual $\mathcal{C}^{k,\alpha}$ norm on $[s, s+1] \times \mathbb{S}^1$.
\end{Def}
A function belonging to $\mathcal{C}^{k,\alpha}_{\mu}([S, \infty) \times \mathbb{S}^1)$ satisfies an estimate of the form
\[ |u|_{k,\alpha, s} \leq c(\cosh s)^{\mu} \hspace{5pt} \mbox{ for all } s \geq S. \]
Thus a positive weight allows functions of limited exponential growth, a negative weight requires functions to decay at a sufficiently fast exponential rate, and $\mathcal{C}^{k,\alpha}_{\mu} \subset \mathcal{C}^{k,\alpha}_{\nu}$ if $\mu \leq \nu$. 

The reason for considering these \emph{weighted} H\"older spaces is that - except for a discrete, countable set of weights - operators such as $\mathcal{L}_{0}$ and $\mathcal{L}_{\e}$ are Fredholm on these spaces. The exceptional set of weights is the sequence of \emph{indicial roots} of the operator. These indicial roots describe all the possible asymptotic growth and decay rates of solutions to the homogeneous equation. There is a fundamental relationship between the weight, indicial roots and Fredholm index, which is an important consideration in our analysis. To better understand this we'll review the relationship in the simplest setting of the flat Laplacian on a half-cylinder. 

\subsection{Weights, indicial roots and Fredholm index: an example}

This example is a slight modification of an example in \cite{PR}. We denote the flat Laplacian on a half-cylinder $[S, \infty) \times \mathbb{S}^1$ by
\[ L_0 := \partial_s^2 + \partial_{\theta}^2. \]
Later we will need to understand a perturbation of $L_0$ to analyse the Jacobi operator.

Throughout we'll work with a fixed orthonormal basis $\{\psi_n\}_{n \in \mathbb{Z}}$ for $L^2(\mathbb{S}^1)$, such that
\[ \partial_{\theta}^2 \psi_n = -n^2\psi_n, \hspace{5pt} \mbox{ for all } n \in \mathbb{Z}. \]

Let 
\[ \mathcal{C}^{2,\alpha}_{\mu, D}([S, \infty) \times \mathbb{S}^1) := \left\{ u \in \mathcal{C}^{2,\alpha}_{\mu}([S, \infty) \times \mathbb{S}^1): u(S,\theta) = 0 \right\}.  \] 

The purpose of this discussion is to illustrate two main facts. First, consider the bounded operator
\begin{equation} L_0:  \mathcal{C}^{2,\alpha}_{\mu, D}([S, \infty) \times \mathbb{S}^1) \longrightarrow \mathcal{C}^{0,\alpha}_{\mu}([S, \infty) \times \mathbb{S}^1). \label{LoWithSpaces} \end{equation}
 As alluded to above, the operator \eqref{LoWithSpaces} is Fredholm, provided the weight $\mu$ is not an indicial root. In Proposition \ref{MappingYoga} we show that the indicial roots are the integers, and prove the Fredholm property in the case most relevant to us, when $\mu < 0$ and the operator is injective. It turns out that the cokernel is always non-trivial when the kernel is trivial. Furthermore the dimension of the cokernel increases each time the weight crosses below a negative indicial root. 

Proposition \ref{MappingYoga} implies that there is no choice of weight for which \eqref{LoWithSpaces} is invertible. In particular, when $\mu < 0$ it is not possible to find a solution $u \in \mathcal{C}^{2,\alpha}_{\mu}([S, \infty) \times \mathbb{S}^1)$ to the boundary value problem
\begin{equation} \left\{ \begin{array}{cc} L_0 u = f & \mbox{ in } [S,\infty) \times \mathbb{S}^1 \\
u(S,\theta) = 0 & \mbox{ on } \{S\} \times \mathbb{S}^1. \end{array} \right. \end{equation}
for every $f \in \mathcal{C}^{0,\alpha}_{\mu}([S, \infty) \times \mathbb{S}^1)$.
However, in Proposition \ref{LoRightInv} we prove that it is possible to partially invert it in the following sense. We can find, for every  $f \in \mathcal{C}^{0,\alpha}_{\mu}([S, \infty) \times \mathbb{S}^1)$, a unique $u \in \mathcal{C}^{2,\alpha}_{\mu}([S, \infty) \times \mathbb{S}^1)$ with the properties that $L_0u = f$ and $u$ vanishes on the boundary, except on the certain 'low modes' $\psi_{-k}, \cdots, \psi_0, \cdots, \psi_k$. The number of these low modes is the dimension of the cokernel of \eqref{LoWithSpaces}. 

In summary then, the closest one can come to inverting $L_0$ in this setting, is with the solution operator $G$ of Proposition \ref{LoRightInv}. This provides solutions to the equation $L_0 u = f$ for every $f$, but the price we pay is that $u$ may not have zero boundary data on the low modes. Finally - and this is the key point - the dimension of this space of low modes increases every time the weight crosses below a negative indicial root.


\begin{Prop}\emph{(\cite{PR})} 
If $\mu \notin \mathbb{Z}$, then
\begin{equation} L_0:  \mathcal{C}^{2,\alpha}_{\mu, D}([S, \infty) \times \mathbb{S}^1) \longrightarrow \mathcal{C}^{0,\alpha}_{\mu}([S, \infty) \times \mathbb{S}^1) \label{LoWithSpaces} \end{equation}
is a Fredholm operator. Let $k$ be the largest integer such that $0 \leq k < |\mu|$. 
\begin{enumerate}[(a)]
\item If $\mu < 0$ and not an integer, the operator is injective with $2k + 1$-dimensional cokernel. 
\item If $\mu > 0$ and not an integer, the operator is surjective with $2k + 1$-dimensional kernel.
\end{enumerate}
\label{MappingYoga}
 \end{Prop}

As we mentioned, we'll only discuss the most relevant case $(a)$. The proof of injectivity is elementary, but the characterisation of the cokernel requires Proposition \ref{LoRightInv} below. 

\begin{proof}[Proof of injectivity]
Assume $\mu < 0$, $\mu \notin \mathbb{Z}$ and $u$ is a solution to the homogeneous equation $L_0u= 0$. Expand $u$ in a Fourier series as
\[ u(s,\theta) = \sum_n u_n(s)\psi_n(\theta). \] Since
\[ L_0u = \sum_{n \in \mathbb{Z}}\left(u_n''(s) - n^2u_n(s)\right)\psi_n(\theta), \]the equation $L_0u = 0$ is equivalent to the family of ordinary differential equations
\[  u_n'' - n^2 u_n = 0, \hspace{5pt}  n \in \mathbb{Z} . \]
The solutions are of course
\[ u_n(s) = \left\{ \begin{array}{cc} a_0 + b_0s & \mbox{ if } n = 0 \\ 
 a_ne^{-|n|s} + b_ne^{|n|s} & \mbox{ if } n \neq 0. \end{array} \right. \]
Admission to the weighted space requires that each solution $u_n$ decay at infinity, because $\mu$ is negative. Hence $b_n = 0$ for all $n \in \mathbb{Z}$, $a_0 = 0$, and
\[ u(s,\theta) = \sum_{n \neq 0}a_ne^{-|n|s}\psi_n(\theta). \]
Applying the boundary condition $u(S,\theta) = 0$ forces the coefficients $a_n$ to vanish, and with them $u$. \qedhere
\end{proof}

\begin{Prop} \label{LoRightInv}
Let $\mu < 0$ and not an integer, and let $k$ be the largest integer such that $0 \leq k < |\mu|$. Then for each $f \in \mathcal{C}^{0,\alpha}_{\mu}([S, \infty) \times \mathbb{S}^1)$ there exists a unique solution $u \in \mathcal{C}^{2,\alpha}_{\mu}([S, \infty) \times \mathbb{S}^1)$ to the boundary value problem
\begin{equation} \left\{ \begin{array}{cc} L_0 u = f & \mbox{ in } [S,\infty) \times \mathbb{S}^1 \\
u(S,\theta) \in span\{\psi_n: |n| \leq k \} & \mbox{ on } \{S\} \times \mathbb{S}^1. \end{array} \right. \label{LoPartialDirichlet} \end{equation} Hence there is a well-defined solution operator $G$ given by 
\[ \mathcal{C}^{0,\alpha}_{\mu}([S, \infty) \times \mathbb{S}^1) \ni f \longmapsto u = G(f) \in \mathcal{C}^{2,\alpha}_{\mu}([S, \infty) \times \mathbb{S}^1). \]
In fact $G$ is bounded, that is, there exists a constant $c > 0$ such that 
\begin{equation} ||u||_{2,\alpha, \mu} \leq c||f||_{0,\alpha, \mu} \hspace{5pt} \mbox{ for all } f \in \mathcal{C}^{0,\alpha}_{\mu}([S, \infty) \times \mathbb{S}^1). \label{Gbddest} \end{equation}
 \end{Prop}

\begin{proof}
We begin with the proof of existence and the estimate \eqref{Gbddest}. Expand $u$ and $f$ in Fourier series as
\[ u(s,\theta) = \sum_{n \in \mathbb{Z}}u_n(s)\psi_n(\theta) \hspace{5pt} \mbox{ and } \hspace{5pt} f(s,\theta) = \sum_{n \in \mathbb{Z}}f_n(s)\psi_n(\theta). \]
The equation we wish to solve is equivalent to the family of ordinary differential equations
\[ u_n''(s) - n^2u_n(s) = f_n(s), \hspace{5pt} n \in \mathbb{Z}. \]
These can be solved using ''variation of coefficients'', since the solutions to the homogeneous equations are explicitly known. As we wish to obtain decaying functions $u_n$, we discard the unbounded solutions to the homogeneous equation, taking the ansatz
\[ u_n(s) = a_n(s)e^{-|n|s}, \hspace{5pt} n \in \mathbb{Z}. \]
It follows from the equation for $u_n$ that $a_n$ must satisfy
\[ a_n''(s)e^{-|n|s} - 2|n|a_n'(s)e^{-|n|s} = f_n(s). \]
This is first order and linear in $a_n'$, so an integrating factor of $e^{-|n|s}$ gives us
\begin{equation} \left( a_n'(s)e^{-2|n|s}\right)' = e^{-|n|s}f_n(s). \label{anone} \end{equation}
To integrate this, notice that the right hand side is $\mathcal{O}(e^{(-|n| + \mu)s})$ as $s \rightarrow \infty$, because $f \in \mathcal{C}^{0,\alpha}_{\mu}([S, \infty) \times \mathbb{S}^1)$ implies that
\[ |f_n(s)| \leq ||f||_{0,\alpha,\mu}e^{\mu s} \mbox{ for all } s \geq S. \]
Therefore we may integrate in from infinity, obtaining 
\[ a_n'(t)e^{-2|n|t} = - \int^{\infty}_te^{-|n|\tau}f_n(\tau) d\tau = \mathcal{O}\left( \frac{||f||_{0,\alpha,\mu}}{|n| - \mu}e^{(-|n| + \mu)t}\right) \mbox{ as } t \rightarrow \infty. \] If instead we had integrated \eqref{anone} from $S$ to $t$, such an estimate need not hold. 
In summary then, 
\begin{equation} a_n'(t) = - e^{2|n|t}\int^{\infty}_te^{-|n|\tau}f_n(\tau) d\tau = \mathcal{O}\left( \frac{||f||_{0,\alpha,\mu}}{|n| - \mu}e^{(|n| + \mu)t}\right) \mbox{ as } t \rightarrow \infty.\label{antwo} \end{equation}

Now there arises a dichotomy. If $|n| < |\mu|$, the right hand side of \eqref{antwo} is exponentially decaying, whilst if $|n| > |\mu|$ it may grow exponentially. In the first case we integrate \eqref{antwo} from $\infty$ to $s$, as before, whilst in the second we must integrate from $S$ to $s$. Therefore our solutions are
\begin{equation} u_n(s) = e^{-|n|s}\int^{\infty}_s e^{2|n|t}\int^{\infty}_te^{-|n|\tau}f_n(\tau) d\tau dt \label{solnlowmodes} \end{equation} if $|n|< |\mu|$ and
\begin{equation} u_n(s) = -e^{-|n|s}\int^s_S e^{2|n|t}\int^{\infty}_te^{-|n|\tau}f_n(\tau) d\tau dt \label{solnhighmodes} \end{equation} 
 if $|n| > |\mu|$. In all cases we have the estimate
\[ |u_n(s)| =  \mathcal{O}\left( \frac{||f||_{0,\alpha,\mu}}{|n^2 - \mu^2|}e^{ \mu s}\right) \mbox{ as } s \rightarrow \infty. \]  
These bounds on $u_n$ are summable in $n$, therefore there exists a constant $c > 0$ independent of $f$ such that
\[ |u(s,\theta)| \leq c||f||_{0,\alpha,\mu}e^{\mu s} \hspace{5pt} \mbox{ for all } (s,\theta) \in [S,\infty)\times \mathbb{S}^1. \]
Hence the full estimate \eqref{Gbddest} follows from Schauder's estimates. 

The proof of uniqueness follows the argument for injectivity in Proposition \ref{MappingYoga} $(a)$. Assume we have a function $u$ satisfying \eqref{LoPartialDirichlet} with $f = 0$, and belonging to $\mathcal{C}^{2,\alpha}_{\mu}([S, \infty) \times \mathbb{S}^1)$. As above we find
\[ u(s,\theta) = \sum_{n \neq 0}(a_ne^{-|n|s} + b_ne^{|n|s})\psi_n(\theta) + a_0 + b_0s. \]
By assumption $u$ is not only decaying, but $\mathcal{O}(e^{\mu s})$, as $s \rightarrow \infty$. This forces $b_n = 0$ for all $n \in \mathbb{Z}$, and $a_n = 0$ for all $0 \leq |n| \leq k$. So
\[ u(s,\theta) = \sum_{|n| > k}a_ne^{-|n|s}\psi_n(\theta), \]
and the boundary condition in \eqref{LoPartialDirichlet} implies that $a_n = 0$ for all $|n| > k$, hence $u = 0$. \qedhere
\end{proof}


\begin{proof}[Proof of Proposition \ref{MappingYoga} (a)]
Given $f \in \mathcal{C}^{0,\alpha}_{\mu}([S, \infty) \times \mathbb{S}^1)$, let $u$ be the unique solution to the boundary value problem \eqref{LoPartialDirichlet} furnished by Proposition \ref{LoRightInv}. Since $\mu$ is negative, the mapping \eqref{LoWithSpaces} is injective, so $u$ is the only candidate for a pre-image of $f$. In fact $u$ provides such a pre-image if and only if it vanishes identically on the boundary, which occurs precisely when $u_n(S) = 0$ for all $|n| < |\mu|$. Intuitively this identifies the cokernel of \eqref{LoWithSpaces} with the space of $2k + 1$-tuples $(u_n(S))_{|n| \leq k}$. Formally, define a mapping
\[ \Lambda: \mathcal{C}^{0,\alpha}_{\mu}([S, \infty) \times \mathbb{S}^1) \ni f \longrightarrow (u_n(S))_{|n| \leq k} \in \mathbb{R}^{2k + 1}. \]
Proposition \ref{LoRightInv} implies that this is well-defined and bounded. Therefore we have a short exact sequence of bounded linear maps
\[ 0 \stackrel{0}{\longrightarrow} \mathcal{C}^{2,\alpha}_{\mu, D}([S, \infty) \times \mathbb{S}^1) \stackrel{L_0}{\longrightarrow} \mathcal{C}^{0,\alpha}_{\mu}([S, \infty) \times \mathbb{S}^1) \stackrel{\Lambda}{\longrightarrow} \mathbb{R}^{2k + 1} \stackrel{0}{\longrightarrow} 0, \]
and the cokernel of \eqref{LoWithSpaces} is isomorphic to $\mathbb{R}^{2k + 1}$. \qedhere
\end{proof}





In our solution of the mean curvature equation by a contraction mapping we must invert the Jacobi operator. As with $L_0$ this is only possible when the weight is not an indicial root, and then only with a solution operator analogous to that of Proposition \ref{LoRightInv}. Before we can prove such a result we must understand the indicial roots of $\mathcal{L}_{\e}$. 


\subsection{Perturbation of spectra and indicial roots}\label{PertSpecRoot}

It turns out that for the linear analysis we can work with an operator that is simpler than even $\mathcal{L}_{\e}$. As we will see, the exponentially decaying term in the potential can be absorbed in 'error' terms, so we will study
\[ L_{\e} := \partial_s^2 + \partial_{\theta}^2 + \alpha^{-2}E. \]
As for the flat Laplacian $L_0$ discussed in the previous section, the indicial roots of $L_{\e}$ are determined by the eigenvalues of the cross-sectional operator. In this case that operator is $\partial_{\theta}^2 + \alpha^{-2}E$ rather than $\partial_{\theta}^2$.


\subsubsection*{Spectrum of $\partial_{\theta}^2 + \alpha^{-2}E$}

Since $\csl$ is a perturbation of $\partial_{\theta}^2$, whose spectrum we know, it is natural to try to apply the perturbation theory of linear operators to understand its spectrum. The results of this theory are most straightforward in the setting of a family of self-adjoint operators on a fixed Hilbert space \cite{K}. Unfortunately $\csl$ is \emph{not} self-adjoint on $L^2(\mathbb{S}^1, d\theta)$. However the Jacobi operator is self-adjoint with respect to the area measure on the catenoids, so $\csl$ ought to be self-adjoint with respect to a suitable measure obtained from the catenoids' metrics. This is indeed the case, the measure being $(1 + \alpha^{-2}\cos^2\theta)^{-1/2}d\theta$. Therefore we can arrange to work with a family $T_{\e}$ of self-adjoint operators on $L^2(\mathbb{S}^1, d\theta)$ by conjugating $\csl$ by the obvious isometry of this Hilbert space and $L^2(\mathbb{S}^1, (1 + \alpha^{-2}\cos^2\theta)^{-1/2}d\theta)$. Of course, $T_{\e}$ and $\csl$ have exactly the same eigenvalues. 

\begin{Def}
Let $M_{\e}$ be the isometry
\[ L^2(\mathbb{S}^1, d\theta) \ni \phi \longmapsto (1 + \alpha^{-2}\cos^2\theta)^{1/4}\phi \in  L^2(\mathbb{S}^1, (1 + \alpha^{-2}\cos^2\theta)^{-1/2}d\theta). \]
We'll consider the family of self-adjoint operators on $L^2(\mathbb{S}^1, d\theta)$ defined by
\[ T_{\e} = M_{\e}^{-1} \circ (\csl) \circ M_{\e}. \]
Note that $T_0 = \partial_{\theta}^2$. 
\end{Def}


There is one final consideration to be made before discussing the perturbation theory. Recall that the horizontal catenoids are invariant under a reflection $\tilde{\rho}$ in the plane $\mathbb{H}^2 \times \{0\}$. Corresponding to this ambient symmetry is a coordinate symmetry $\theta \mapsto \pi - \theta$. That is, the embeddings $\mathbf{X}_{\alpha}$ of the catenoids satisfy
\[ \tilde{\rho} \left(\mathbf{X}_{\alpha}(s,\theta) \right) = \mathbf{X}_{\alpha}(s,\pi - \theta). \]
Therefore functions on the catenoids are invariant under vertical reflection if, for each fixed $s$, they belong to the following space. 
\begin{Def}
Let 
\[L^2_e(\mathbb{S}) := \{\phi \in L^2(\mathbb{S}): \phi \circ \rho = \phi\}, \hspace{5pt} \mbox{ where }  \rho(\theta) = \pi - \theta. \]
\end{Def}  
There's a good geometric reason for considering such functions, which we explain below when we discuss geometric Jacobi fields. On the other hand much of our analysis does not depend on this symmetry. Nevertheless we will go ahead and impose it on our function spaces, to avoid a proliferation of notations. 

The space $L^2_e(\mathbb{S})$ has an orthonormal basis of eigenfunctions for $\partial_{\theta}^2$, consisting of $\{\psi_n\}_{n = 0}^{\infty}$, where 
\[ \psi_n(\theta) := \left\{ \begin{array}{cc} \frac{1}{\sqrt{2\pi}} & \mbox{ $n = 0$ } \\ \vspace{5pt} \frac{\cos(n\theta)}{\sqrt{\pi}} & \mbox{ $n$  even, non-zero } \\ \vspace{5pt} \frac{\sin(n\theta)}{\sqrt{\pi}} & \mbox{ $n$ odd, non-zero. } \end{array} \right. \]
From this description it follows that the operators $M_{\e}$, $\csl$ and $T_{\e}$ reduce to this subspace, so we can and will consider $T_{\e}$ as a family of self-adjoint operators on $L^2_e(\mathbb{S}, d\theta)$. 

On $L^2_e(\mathbb{S})$ the eigenvalues of $\partial_{\theta}^2$ are all simple; we'll denote them by $\lambda_n = -n^2, n \geq 0$. Therefore the eigenvalues of $\csl$ on $L^2_e(\mathbb{S})$ are also simple, and we denote them by $\lambda_{n,\e}, n \geq 0$. 

\begin{Lem}\label{lambdatwo} For every $n \geq 2$,
\[ \lambda_{n,\e} = -n^2 - \frac{n^2}{2}\alpha^{-2} + \mathcal{O}(\alpha^{-4}). \]
\end{Lem}

\begin{proof}
We invoke the perturbation theory of analytic families of self-adjoint operators (see \cite{K}, Chapter $7$, section $3$). It tells us that the eigenvalues and eigenvectors of $T_{\e}$ may be chosen to have (convergent) series expansions in powers of $\alpha^{-2}$ in such a way that $\psi_n$ is an eigenfunction for the unperturbed operator $\partial_{\theta}^2$ corresponding to the eigenvalue $\lambda_n$:
\begin{eqnarray*}
\psi_{n,\e}(\theta) & = & \psi_n(\theta) + \alpha^{-2}\psi_{n}^{(1)}(\theta) + \alpha^{-4}\psi_{n}^{(2)}(\theta) + \cdots \\
\lambda_{n,\e} & = & \lambda_n + \alpha^{-2}\lambda_n^{(1)} + \alpha^{-4}\lambda_n^{(2)} + \cdots.
\end{eqnarray*}

 The eigenfunctions for $\csl$ are then $M_{\e}\psi_{n,\e} = (1 + \alpha^{-2}\cos^2\theta)^{1/4}\psi_{n,\e}$; clearly these may also be expanded in a series in powers of $\alpha^{-2}$. In particular we have  
\[ M_{\e}\psi_{n,\e} = \psi_n(\theta) + \alpha^{-2}\psi(\theta) + \mathcal{O}(\alpha^{-4}) \hspace{5pt} \mbox{ and } \hspace{5pt} 
\lambda_{n,\e} = -n^2 + \alpha^{-2}\lambda + \mathcal{O}(\alpha^{-4}), \] and we wish to find $\lambda$. 
%
%
%
%
Substituting these into the eigenvalue equation $(\csl)M_{\e}\psi_{n,\e} = \lambda_{n,\e}M_{\e}\psi_{n,\e}$ and equating coefficients of $\alpha^{-2}$, we get
\[ (\partial_{\theta}^2 + n^2)\psi = (-E + \lambda)\psi_n. \]
Now, the left hand side is orthogonal to $\psi_n$ in $L^2(\mathbb{S}, d\theta)$, hence the coefficient of $\psi_n$ on the right hand side must vanish. We claim that the only contribution of $E\psi_n$ to the $\psi_n$ Fourier mode is $-\frac{n^2}{2} \psi_n$, from which the lemma then follows. 

To prove the claim, write
\[ E\psi_n = \frac{1}{2}\partial_{\theta}^2\psi_n + \frac{\cos(2\theta)}{2}\partial_{\theta}^2\psi_n - \frac{\sin(2\theta)}{2}\partial_{\theta}\psi_n + \cos(2\theta)\psi_n \]
The first term on the right hand side is the $-\frac{n^2}{2} \psi_n$ term. All other terms involve a factor of $\cos(2\theta)$ or $\sin(2\theta)$ multiplying $\psi_n$ or its first or second derivative. By standard identities for products of the form $\cos(m\theta)\cos(n\theta), \sin(m\theta)\sin(n\theta)$ and $\sin(m\theta)\cos(n\theta)$, those terms only involve $\psi_{n-2}$ and $\psi_{n + 2}$. Since $n \geq 2$ the functions $\psi_{n+2}$ and $\psi_{n-2}$ are distinct from $\psi_n$, and the claim is proved. \qedhere
 \qedhere
\end{proof} 

For the low eigenvalues we have more precise information. 

\begin{Lem}\label{LowEval}
The first three eigenvalues of $\csl$ on $L^2(\mathbb{S})$ (not $L^2_e(\mathbb{S})$) are $0, -1 $ and $-(1 + \e)^2$.  Their eigenspaces are spanned by $(1 + \alpha^{-2}\cos^2 \theta)^{1/2}, \cos \theta$ and $\sin \theta$ respectively. Hence the first two eigenvalues of $\csl$ on $L^2_e(\mathbb{S})$ are $\lambda_{0,\e} = 0$ and $\lambda_{1,\e} = -(1 + \e)^2$.  
\end{Lem}

\begin{proof}
The proof is by direct calculation; note that $E \cos \theta = 0$, $E \sin \theta = -2\sin \theta$ and $1 + \alpha^{-2} = (1 + \e)^2$ by \eqref{epsilon}. \qedhere
\end{proof}

In other words, on the full space $L^2(\mathbb{S})$, the multiplicity two eigenvalue $-1$ for $\partial_{\theta}^2$ splits into two simple eigenvalues $-1$ and $-(1 + \e)^2$, but on the restricted space $L^2_e(\mathbb{S})$ the eigenvalue $-1$ is simple and is perturbed to $-(1 + \e)^2$. 




\subsubsection*{Indicial roots}

As in the discussion of the indicial roots of $L_0$, the indicial roots of $L_{\e}$ are essentially the square-root 
of the eigenvalues of the cross-sectional operator $\csl$. Just as the eigenvalues of $\csl$ depend on whether or not we impose the symmetry discussed above, so do the indicial roots. Therefore we make the following \emph{convention}: 
\begin{quote}
Unless otherwise stated, when we refer to 'the indicial roots of $L_{\e}$' it refers to the operator acting on spaces of functions that are invariant under the vertical reflection symmetry discussed above. 
\end{quote}

By the same argument as described for $L_0$ at the beginning of section \ref{PertSpecRoot}, and the results of Lemmas \ref{lambdatwo} and \ref{LowEval}, we have the result we will need on the indicial roots. 

\begin{Lem}\label{IndicialRoots}
The indicial roots $\{\gamma_{n,\e}\}_{n \in \mathbb{Z}}$ of $L_{\e}$ are given by $\gamma_{\pm n,\e} = \pm \sqrt{|\lambda_{n,\e}|}$ for all $n \geq 0$. Hence, for all small $\e > 0$, we have
\[ 0 = \gamma_{0,\e} < \gamma_{1,\e} = 1 + \e  < 2 < \gamma_{2,\e} < \cdots \]
and 
\[ \gamma_{n,\e} > n \mbox{ for all } n \geq 2. \] 
\end{Lem}

\subsection{Mapping properties on weighted H\"older spaces}

In this section we will establish the linear analysis that we will use to prove our main theorem. As discussed above, we'll analyse the operator $L_{\e}$ on weighted H\"older spaces, and we wish to begin by discussing our choice of weight. The most basic requirement is that the weight be negative, so that the normal graphs are decaying and define surfaces that are asymptotic to the ends of horizontal catenoids. Beyond this there are two competing factors that motivate the choice of weight. 

The first factor is that when we 'invert' the operator $L_{\e}$, we would like to be able to prescribe as much of the boundary data to be zero as possible. As discussed following Proposition \ref{MappingYoga}, when the weight is negative we can only prescribe zero boundary data up to a finite-dimensional subspace of the full set of boundary values. Each time we decrease the weight across an indicial root the dimension of this subspace increments. Thus, ideally, we would like to use a weight below zero but above the first negative indicial root. 

The second factor determining our choice of the weight is the application to the mean curvature equation. Propositions \ref{MeanCurvExp} and \ref{JacobiExactProp} imply that the normal graph defined by $w$ has cmc $1/2$ if and only if
\[ L_{\e}w = (\partial_s^2 + \partial_{\theta}^2  + \alpha^{-2}E)w = - 2(\cosh s)^{-2}w + \e^2\cosh^2s Q_0(w). \] Our strategy to solve this equation is to 'invert' $L_{\e}$ and reformulate it as a fixed point problem, formally
\[ w = L_{\e}^{-1}\left(- 2(\cosh s)^{-2}w + \e^2\cosh^2s Q_0(w)\right). \]
This requires that the right hand side belong to the same weighted space as $w$ originally does. If $w$ decays with weight $\mu$, then the right hand side belongs to the space with weight $2 + 2\mu$. (Recall the notation $Q_0$ means that there is no growth in the coefficients of the operator as $s \rightarrow \infty$.) So our requirement amounts to $2 + 2\mu \leq \mu$ and therefore $\mu \leq -2$. 

There is a tension between choosing the weight sufficiently negative, but minimising the number of negative indicial roots crossed. Crossing the first negative indicial root $\gamma_{-1,\e}$  is unavoidable as $\gamma_{-1,\e} = -1 - \e$ (by Lemma \ref{IndicialRoots}). On the other hand the second negative indicial root $\gamma_{-2,\e} < -2$. Therefore we can work with a weight of $-2$ without crossing $\gamma_{-2,\e}$, and we will do so.

\begin{Def}
Let $\{\psi_{n,\e}\}_{n=0}^{\infty}$ be an orthonormal basis of eigenfunctions for $\csl$. We define the projections $\pi_{\e}''$ by
\[ \pi''_{\e}\left( \sum_{n=0}^{\infty}a_n\psi_{n,\e}\right) = \sum_{n=2}^{\infty}a_n\psi_{n,\e}. \]
\end{Def}

Now we prove the two main results of this section, which provide solutions to homogeneous and inhomogeneous boundary value problems for $L_{\e}$ in weighted H\"older spaces on the ends of horizontal catenoids (as defined in section \ref{truncation}). 

\begin{Prop}\label{PoissonOp}
There exists a bounded linear operator
\[ \mathcal{P} = \mathcal{P}_{\e}: \pi''_{\epsilon} \mathcal{C}^{2,\alpha}(\mathbb{S}^1) \rightarrow \mathcal{C}^{2,\alpha}_{-2}\left([S_{\e}, \infty) \times \mathbb{S}^1 \right) \]
such that for each $\phi \in \pi''_{\epsilon} \mathcal{C}^{2,\alpha}(\mathbb{S}^1)$, the function $u := \mathcal{P}_{\e}(\phi)$ is the unique solution in $\mathcal{C}^{2,\alpha}_{-2}\left([S_{\e}, \infty) \times \mathbb{S}^1 \right)$ of the boundary value problem
\[ \left\{ \begin{array}{ccc} L_{\e} u = 0 & \mbox{in} & (S_{\e},\infty) \times \mathbb{S}^1 \\
u(S_{\e}, \theta) = \phi(\theta) & \mbox{on} & \{S_{\e}\} \times \mathbb{S}^1. \end{array} \right. \]
Furthermore there exists a constant $c > 0$ independent of $\e$ such that
\[ ||u||_{2,\alpha, -2} \leq ce^{2 S_{\e}}||\phi||_{2,\alpha}. \]
\end{Prop}

\begin{proof}
We expand $u$ in terms of eigenfunctions for $\csl$ as
\[ u(s,\theta) = \sum_{n = 0}^{\infty}u_n(s) \psi_{n,\epsilon}(\theta). \]
Then the equation $L_{\epsilon}u = 0$ is equivalent to the family of ordinary differential equations
\[ u_n'' + \lambda_{n,\epsilon} u_n = 0 \hspace{15pt}  n = 0,1,2,\cdots. \]
Since $\lambda_{0,\e} = 0$ and $\lambda_{n,\e} = - \gamma_{n,\e}^2$ for $n \geq 1$ these have solutions 
\[ u_0(s) = a_0 + b_0s \hspace{5pt} \mbox{ and } \hspace{5pt} a_ne^{-\gamma_{n,\e}s} + b_ne^{\gamma_{n,\e}s} \mbox{ for } n\geq 1. \]
Imposing the decay condition forces $u_0 = u_1 \equiv 0$ and $b_n = 0$ for all $n$, because we require $u$ to decay at least as fast as $(\cosh s)^{-2}$, and by Lemma \ref{IndicialRoots} we know $\gamma_{-n,\e} = -\gamma_{n,\e}$ for $n \geq 1$, whilst
\[ \gamma_{1,\e} = 1 + \e < 2 < \gamma_{2,\e} < \gamma_{3,\e} \cdots. \]  Imposing the boundary condition we have the solution
\[ u(s,\theta) = \sum_{n=2}^{\infty} \phi_{n,\e}e^{-\gamma_{n,\e}(s - S_{\e})}\psi_{n,\e}(\theta), \]
where $\phi_{n,\e}$ are defined by $\phi = \sum_{n \geq 2} \phi_{n,\e}\psi_{n,\e}$. 

To obtain the weighted H\"older estimate for $u$ we begin with a weighted $C^0$ estimate and then invoke Schauder's estimates. 
Consider first when $s \geq S_{\e} + 1$; then
\[ |u(s,\theta)| \leq \sum_{n=2}^{\infty}|\phi_{n,\e}|e^{-\gamma_{n,\e}(s - S_{\e})} 
 \leq ||\phi||_{2,\alpha}\sum_{n=2}^{\infty}e^{-\gamma_{n,\e}(s - S_{\e})} \leq c||\phi||_{2,\alpha} e^{-2(s - S_{\e})}. \]
Here we have used the fact that $\gamma_{n,\e} \geq n$ for all $n \geq 2$ to compare the series in the last inequality with a geometric series. From this we get
\[ e^{2s}|u(s,\theta)| \leq ||\phi||e^{2S_{\e}} \hspace{5pt} \mbox{ for all } s \geq S_{\e} + 1. \]
In the domain $S_{\e} \leq s \leq S_{\e} + 1$ we only require the estimate
\begin{equation}
|u(s,\theta)| \leq c||\phi||_{2,\alpha};
\label{BoundEstimate}
\end{equation}
the weighted estimate follows since $e^{2s} \leq ce^{2S_{\e}}$ here. The estimate \eqref{BoundEstimate} in turn follows from the fact that
\begin{equation} |\phi_{n,\e}| \leq c|\lambda_{n,\e}|^{-1}||\phi||_{2,\alpha} \hspace{5pt} \mbox{ for all } n \geq 1. \label{FourierCoeffEst}
\end{equation}
Recall that  $\csl$ is self-adjoint with respect to $(1 + \alpha^{-2}\cos^2 \theta)^{-1/2}d\theta$. We write the inner product and norm on this Hilbert space as
\[ \la f,g \ra_{\e} := \int_{\mathbb{S}^1}f(\theta) g(\theta)(1 + \alpha^{-2}\cos^2 \theta)^{-1/2}d\theta, \hspace{5pt} \mbox{ and } \hspace{5pt}  ||f||_{\e} = \sqrt{\la f,f\ra_{\e}}, \]
\[ \mbox{  so } \hspace{5pt} |\phi_{n,\e}| = \left|\la \phi, \psi_{n,\e} \ra_{\e}\right| = |\lambda_{n,\e}|^{-1}\left|\la \phi, (\csl)\psi_{n,\e} \ra_{\e}\right|.  \] Therefore self-adjointness and Cauchy-Schwarz' inequality  give 
\[ |\phi_{n,e}| = |\lambda_{n,\e}|^{-1}\left|\la (\csl)\phi, \psi_{n,\e} \ra_{\e}\right| \leq |\lambda_{n,\e}|^{-1}||(\csl)\phi||_{\e} \leq c|\lambda_{n,\e}|^{-1}||\phi||_{2,\alpha}. \]
This proves \eqref{FourierCoeffEst}, and \eqref{BoundEstimate} follows because $ \sum_{n \geq 1} |\lambda_{n,\e}|^{-1} \leq c$ by Lemma \ref{lambdatwo}. 

This completes the proof of the weighted $C^0$ estimate 
\[ e^{2s}|u(s,\theta)| \leq ce^{2S_{\e}}||\phi||_{2, \alpha} \hspace{5pt} \mbox{ for all } s \geq S_{\e}, \]
and the full weighted $C^{2,\alpha}$ estimate follows from this and Schauder's boundary and interior estimates. \qedhere
\end{proof}

\begin{Prop}\label{GreenOp}
There exists a bounded linear operator
\[ \mathcal{G} = \mathcal{G}_{\e}: \mathcal{C}^{0,\alpha}_{-2}\left([S_{\e}, \infty) \times \mathbb{S}^1 \right)  \rightarrow \mathcal{C}^{2,\alpha}_{-2}\left([S_{\e}, \infty) \times \mathbb{S}^1 \right) \] such that, for each $f \in \mathcal{C}^{0,\alpha}_{-2}\left([S_{\e}, \infty) \times \mathbb{S}^1 \right)$, the function $u = \mathcal{G}(f)$ solves the boundary value problem
\begin{equation}
\left\{ \begin{array}{ccc} L_{\e} u = f & \mbox{in} & (S_{\e},\infty) \times \mathbb{S}^1 \\
\pi''_{\e}u(S_{\e}, \theta) = 0 & \mbox{on} & \{S_{\e}\} \times \mathbb{S}^1. \end{array} \right. 
\end{equation}
Furthermore there exists a constant $c > 0$ independent of $\e$ such that
\[ ||u||_{2,\alpha, -2} \leq c\e^{-1}||f||_{0,\alpha,-2}. \]
\end{Prop}

\begin{proof}
Expand $u$ and $f$ as
\[ u(s,\theta) = \sum_{n=0}^{\infty}u_n(s)\psi_{n,\epsilon}(\theta), \hspace{15pt} f(s,\theta) = \sum_{n=0}^{\infty}f_n(s)\psi_{n,\epsilon}(\theta). \]
The equation $L_{\epsilon}u = f$ reduces to the family of ordinary differential equations
\[ u_n'' - \gamma_{n,\epsilon}^2u_n = f_n. \hspace{15pt}  n = 0,1,2,\cdots. \]
When $n = 0,1$ we place no boundary condition on $u_n$ at $S_{\e}$, but when $n \geq 2$ we require $u_n(S_{\e}) = 0$. 

The solutions $u_n, n \geq 2,$ can be obtained via variation of parameters. In other words, since the solutions of the homogeneous equation are $e^{\pm \gamma_{n,\epsilon}s}$, we take the ansatz
\[ u_n(s) = a_n(s)e^{- \gamma_{n,\epsilon}s} + b_n(s)e^{\gamma_{n,\epsilon}s}. \]
However, since we wish to obtain a decaying solution, we set $b_n \equiv 0$. Our equation for $u_n$ implies the following equation for $a_n$
\[ a_n''(s) - 2\gamma_{n,\epsilon}a_n'(s) = e^{\gamma_{n,\epsilon}s}f_n(s). \]
This is a linear, first order equation in $a_n'$, and we may solve it using integrating factors. In the end we get
\[ u_n(s) = -e^{-\gamma_{n,\epsilon}s}\int_S^s e^{2\gamma_{n,\epsilon}t} \int^{\infty}_t e^{-\gamma_{n,\epsilon}\tau}f_n(\tau) d\tau dt. \]

The $C^0$ estimate for $u_n$ is obtained in a few steps. Firstly, since  $|f_n(\tau)| \leq ||f||_{0,\alpha,-2}e^{-2\tau}$, we have
\[ \left|\int^{\infty}_t e^{-\gamma_{n,\e}\tau}f_n(\tau) d\tau \right| \leq 
\frac{||f||_{0,\alpha,-2}}{\gamma_{n,\e} + 2}e^{-(\gamma_{n,\e} + 2)t}. \]
Hence
\begin{eqnarray*} \left| \int_{S_{\e}}^s e^{2\gamma_{n,\e}t} \int^{\infty}_t e^{-\gamma_{n,\e}\tau}f_n(\tau) d\tau dt \right| & \leq & 
\frac{||f||_{0,\alpha,-2}}{\gamma_{n,\e}^2 - 4}e^{(\gamma_{n,\e} - 2)s}, 
\end{eqnarray*}
and
\[ e^{2s}|u_n(s)| \leq \frac{||f||_{0,\alpha,-2}}{\gamma_{n,\e}^2 - 4} \hspace{10pt} \mbox{ for all } n \geq 2 \mbox{ and } s \geq S_{\e}. \]
Similar arguments show that
\[ u_0(s) = \int_s^{\infty}\int_t^{\infty}f_0(\tau)d\tau dt, \hspace{5pt} \mbox{ and } \hspace{5pt} u_1(s) = e^{-\gamma_{1,\e}s}\int^{\infty}_se^{2\gamma_{1,\e}t} \int_t^{\infty} e^{-\gamma_{1,\e}\tau} f_1(\tau) d\tau dt. \]
These satisfy the weighted $C^0$ estimates
\[ e^{2s}|u_0(s)| \leq \frac{||f||_{0,\alpha,-2}}{4} \hspace{5pt} \mbox{ and } \hspace{5pt} e^{2s}|u_1(s)| \leq \frac{||f||_{0,\alpha,-2}}{4 - \gamma_{1,\e}^2} \mbox{ for all } s \geq S. \]
Now it follows from the weighted $C^0$ estimates and Schauder's estimates that there is a constant $c > 0$ independent of $n$ such that
\[ e^{2s}|u_0(s)\psi_{0,\e}(\theta)|_{2,\alpha,s} \leq c\frac{||f||_{0,\alpha,-2}}{4}, \hspace{10pt} e^{2s}|u_1(s)\psi_{1,\e}(\theta)|_{2,\alpha,s} \leq c\frac{||f||_{0,\alpha,-2}}{4 \gamma_{1,\e}^2} \mbox{ and } \]
\[ e^{2s}|u_n(s)\psi_{n,\e}(\theta)|_{2,\alpha,s} \leq c\frac{||f||_{0,\alpha,-2}}{\gamma_{n,\e}^2 - 4} \hspace{5pt} \mbox{ for $n \geq 2$}. \]
Therefore
\[ ||u||_{2,\alpha,-2} \leq c||f||_{0,\alpha,-2}\sum_{n=0}^{\infty} |\gamma_{n,\e}^2 - 4|^{-1}. \]
By Lemma \ref{lambdatwo} the sum over all $n \neq 2$ is bounded independently of $\e$, whilst $|\gamma_{2,\e}^2 - 4| = \frac{1}{2}\alpha^{-2} + \mathcal{O}(\alpha^{-4})$. Therefore the full sum is $\mathcal{O}(\e^{-1})$ and the estimate is complete. \qedhere
\end{proof}

\subsection{Geometric Jacobi fields}

In the next section we will complete our construction of a space of annular cmc $1/2$ surfaces asymptotic to the ends of horizontal catenoids.
The space will be parametrised by the boundary data we are able to prescribe when solving the mean curvature equation. As our solution to the mean curvature equation is perturbative, this is determined by the boundary data we can prescribe for the linearised equation. Propositions \ref{PoissonOp} and \ref{GreenOp} only allow us to prescribe the boundary data on the ''high modes'', that is, on the image of $\pi''_{\e}$. Thus we have no analytic control over the boundary data on the low modes $\psi_{0,\e}$ and $\psi_{1,\e}$ that span the kernel of $\pi''_{\e}$. 

It turns out, however, that one has some \emph{geometric} control over these low modes, thanks to certain Jacobi fields. These Jacobi fields are solutions to the homogeneous equation $\mathcal{L}_{\e}u = 0$ which arise from geometric deformations of the horizontal catenoids through families of cmc $1/2$ surfaces. Any such deformation has a variation vector field whose normal component gives a Jacobi field. There are five independent Jacobi fields that arise in this way - four from the four-dimensional isometry group of $\hr$, and one from varying the necksize parameter in the family of horizontal catenoids. We call these \emph{geometric Jacobi fields}. 

In principle, one can compute explicit formulae for these Jacobi fields by computing the variation vector field of each of these deformations, and taking its inner product with the normal on the horizontal catenoids. In practice, however, the formulae one obtains are very complicated. However if one works directly with the equation $\mathcal{L}_{\e}u = 0$ one obtains fairly explicit, and rather simple Jacobi fields on the low modes.  

\begin{Lem} \label{KnownJacobis}
The Jacobi fields on the low modes are given by the functions
\[ (1 + \alpha^{-2}\cos^2\theta)^{1/2}\tanh s \hspace{5pt} \mbox{ and } \hspace{5pt} (1 + \alpha^{-2}\cos^2\theta)^{1/2}(s\tanh s - 1), \]
\[ \frac{\cos \theta}{\cosh s} \mbox{ and }  \left(\frac{s}{\cosh s} + \sinh s\right)\cos \theta,  \]
and 
\[  v_{\e}^-(s)\sin \theta  \hspace{5pt} \mbox{ and } \hspace{5pt}  v_{\e}^+(s)\sin \theta.  \]
Here $v_{\e}^{\pm}$ are solutions to the ode
\[ v_{\e}'' + \left( \frac{2}{\cosh^2s} - (1 + \e)^2\right)v_{\e} = 0 \]
such that $v_{\e}^-$ is even, and decays at the rate $e^{-(1 + \e)s}$ as $s \rightarrow \infty$, whilst $v_{\e}^+$ is odd, and grows at the rate $e^{(1 + \e)s}$ as $s \rightarrow \infty$.
\end{Lem} 

\begin{proof}
The proof that these functions satisfy $\mathcal{L}_{\e}u = 0$ is a simple calculation using Lemma \ref{LowEval} and the well-known geometric Jacobi fields of Euclidean catenoids. \qedhere
\end{proof}

The analogy with Euclidean catenoids, and an examination of the zero sets and decay rates of geometric Jacobi fields, strongly suggests that all but one of these solutions are in fact a geometric Jacobi field. We suspect that 
\[  (1 + \alpha^{-2}\cos^2\theta)^{1/2}\tanh s \hspace{5pt} \mbox{ and } \hspace{5pt} (1 + \alpha^{-2}\cos^2\theta)^{1/2}(s\tanh s - 1) \]
arise from hyperbolic dilations, and varying the necksize parameter, respectively, and that
\[  v_{\e}^-(s)\sin \theta  \hspace{5pt} \mbox{ and } \hspace{5pt}  v_{\e}^+(s)\sin \theta.  \]
arise from hyperbolic translations and rotations, respectively. On the other hand, the proof of Lemma \ref{seccurvcoord} computes the Jacobi field arising from vertical translations to be exactly $\frac{\cos \theta}{\cosh s}$. The only unexplained Jacobi field is $\left(\frac{s}{\cosh s} + \sinh s\right)\cos \theta$ -  we are not aware of any geometric deformation that could be responsible for this Jacobi field. 
 
If one grants that these Jacobi fields are indeed the geometric Jacobi fields which we claim they are, then the next result shows that the horizontal catenoids are non-degenerate, in the sense that there are no decaying Jacobi fields that do not arise from geometric deformations.
\begin{Thm}
Let $\mu < 0$, then the Jacobi operator
\[ \mathcal{L}_{\e}: \mathcal{C}^{2,\alpha}_{\mu}\left( \mathbb{R} \times \mathbb{S}^1\right) \rightarrow \mathcal{C}^{0,\alpha}_{\mu}\left( \mathbb{R} \times \mathbb{S}^1\right) \]has a two-dimensional kernel spanned by 
\[ \frac{\cos \theta}{\cosh s} \mbox{ and } v_{\e}^-(s)\sin \theta. \] (Here we do not impose any symmetry requirements on the functions in these spaces).
\end{Thm}

\begin{proof}
Suppose that $\mathcal{L}_{\e} u = 0$, for $u \in \mathcal{C}^{2,\alpha}_{\mu}\left( \mathbb{R} \times \mathbb{S}^1\right)$. Expanding $u(s,\theta)$ as before, as
\[ u(s,\theta) = \sum_n u_n(s)\psi_{n,\e}(\theta), \] the pde reduces to the odes
\[ u_n''  + \left(\lambda_{n,\e} + \frac{2}{\cosh^2s}\right)u_n = 0, \mbox{ for all } n. \]
Now for all $|n| \geq 2$ we multiply through by $u_n'$ and integrate over $\R$. Given the exponential decay of $u$ we may use integration by parts to find
\[ \int_{\R} |u'_n(s)|^2 + \left( |\lambda_{n,\e}| -  \frac{2}{\cosh^2s}\right) |u_n(s)|^2 ds = 0. \]
By Lemma \ref{lambdatwo} we have $|\lambda_{n, \e}| \geq 4$ for all $n \geq 2$, therefore the integrand is non-negative, and $u_n$ vanishes identically for all such $n$.

When $n =0,1$, the functions $u_n$ are linear combinations of the known Jacobi fields of Lemma \ref{KnownJacobis}. Only two of these Jacobi fields are exponentially decaying on the ends, namely $\frac{\cos \theta}{\cosh s}$ and $v_{\e}^-(s)\sin \theta$, and this proves the theorem. \qedhere
\end{proof}

Lastly we would like to motivate why we introduced the symmetry requirement in the results of the linear analysis in the preceding sections. To produce a family of surfaces that can be used in a gluing construction, one would like to have two free parameters for each of the low modes, these parameters being determined by geometric manipulations. Intuitively, two parameters is the correct number, as one wishes to match the Dirichlet and Neumann data on each mode for a smooth gluing. With each of the modes $(1 + \alpha^{-2}\cos^2\theta)^{1/2}$ and $\sin \theta$, there are two geometric deformations that provide the two parameters. However, on the $\cos \theta$ mode the absence of a second geometric deformation leaves us a parameter short of being able to specify the Cauchy data.

Now $\cos \theta$ is not invariant under reflection in the horizontal plane $\mathbb{H}^2 \times \{0\}$. So by imposing this reflection symmetry we can eliminate the degeneracy on this mode. Bearing in mind future applications to gluing, this was our reason for requiring this symmetry in our results above.

\newpage

\section{Construction of cmc $\frac{1}{2}$ ends asymptotic to horizontal catenoids}

With Propositions \ref{MeanCurvExp}, \ref{PoissonOp} and \ref{GreenOp} in place, the proof of our main result is a straightforward matter. 

\begin{Thm}\label{FamCMCEnd}
There exists $\epsilon_0 > 0$, such that for all $\epsilon \in (0, \epsilon_0)$ the following holds. If $\phi \in \pi''_{\epsilon}\mathcal{C}^{2,\alpha}(\mathbb{S}^1)$  and $||\phi|| \leq \epsilon^2$, then there is a solution $w \in \mathcal{C}^{2,\alpha}_{-2}\left([S_{\epsilon}, \infty)\times \mathbb{S}^1\right)$ to the boundary value problem
\[ \left\{ \begin{array}{cc} \mathcal{M}_{\epsilon}(w) = 1/2 & \mbox{ in } (S_{\epsilon}, \infty)\times \mathbb{S}^1 \\
\pi''_{\epsilon} w = \phi & \mbox{ on } \{S_{\epsilon}\} \times \mathbb{S}^1 \end{array} \right. \]
\end{Thm}

\begin{proof}
By Propositions \ref{MeanCurvExp} and \ref{JacobiExactProp} the boundary value problem we wish to solve takes the form
\[ \left\{ \begin{array}{cc} L_{\e}w = -2(\cosh s)^{-2}w + \e^2 \cosh^2s Q_0(w) &  \mbox{ in } (S_{\epsilon}, \infty)\times \mathbb{S}^1 \\
\pi''_{\epsilon} w = \phi & \mbox{ on } \{S_{\epsilon}\} \times \mathbb{S}^1. \end{array} \right. \]
We take as an approximate solution the function $w_0 := \mathcal{P}(\phi)$ furnished by Proposition \ref{PoissonOp}. Thus $w_0 \in \mathcal{C}^{2,\alpha}_{-2}\left([S_{\e}, \infty) \times \mathbb{S}^1 \right)$ satisfies
\[ \left\{ \begin{array}{cc} L_{\e}w_0 = 0 &  \mbox{ in } (S_{\epsilon}, \infty)\times \mathbb{S}^1 \\
 w_0 = \phi & \mbox{ on } \{S_{\epsilon}\} \times \mathbb{S}^1. \end{array} \right. \]
Now we look for a solution of the form $w = w_0 + v$, where we require $v \in \mathcal{C}^{2,\alpha}_{-2}\left([S_{\e}, \infty) \times \mathbb{S}^1 \right)$ to solve
\[ \left\{ \begin{array}{cc} L_{\e}v = -2(\cosh s)^{-2}(w_0 + v) + \e^2 \cosh^2s Q_0(w_0 + v) &  \mbox{ in } (S_{\epsilon}, \infty)\times \mathbb{S}^1 \\
\pi''_{\epsilon} v = 0 & \mbox{ on } \{S_{\epsilon}\} \times \mathbb{S}^1. \end{array} \right. \]
Using the 'inverse' $\mathcal{G}_{\e}$ to $L_{\e}$ established in Proposition \ref{GreenOp}, this boundary value problem for $v$ is equivalent to $v$ being a fixed point for the operator 
\[ \mathcal{N}_{\e}(v) = \mathcal{G}_{\e}\left( -2(\cosh s)^{-2}(w_0 + v) + \e^2 \cosh^2s Q_0(w_0 + v)\right). \]
That such a fixed point exists (and is unique), for all sufficiently small $\e > 0$, follows from the next two inequalities. We claim there exist constants $c_0, c_1 > 0$ independent of $\e$ such that for small $\e > 0$,
\begin{eqnarray}
||\mathcal{N}_{\e}(0)||_{2,\alpha,-2} & \leq &  c_0 \e \hspace{5pt} \mbox{ and } \label{NZero} \\
||\mathcal{N}_{\e}(v) - \mathcal{N}_{\e}(\tilde{v})||_{2,\alpha,-2} & \leq & c_1\e||v - \tilde{v}||_{2,\alpha,-2}, \label{NLip}
\end{eqnarray}
for all $v,\tilde{v} \in B_{2c_0\e}(0) \subset \mathcal{C}^{2,\alpha}_{-2}\left([S_{\e}, \infty) \times \mathbb{S}^1 \right)$. Indeed, it follows from these that $\mathcal{N}_{\e}$ is a contraction map, mapping the ball $B_{2c_0\e}(0)$ into itself, for all sufficiently small $\e$.

Now we turn to the estimates \eqref{NZero}, \eqref{NLip}. First note that $||w_0||_{2,\alpha, -2} \leq c\e^{-2}||\phi||_{2,\alpha}$ by Proposition \ref{PoissonOp}, because $e^{S_{\e}} = \mathcal{O}(\e^{-1})$. Since $||\phi|| \leq \e^2$, we get $||w_0||_{2,\alpha, -2} \leq c$. For the rest of the proof we'll just write $||\cdot||$ in place of the weighted norms $||\cdot||_{0,\alpha,-2}$ and $||\cdot||_{2,\alpha,-2}$; it should be clear which norm is used in each term.

The estimate in Proposition \ref{GreenOp} and the fact that $(\cosh s)^{-1} \leq \e$ on the end imply
\[ ||\mathcal{N}_{\e}(0)|| \leq c\e^{-1}\left( ||(\cosh s)^{-2}w_0|| + \e^2||\cosh^2s Q_0(w)|| \right)  \leq c\e^{-1}\left( \e^2||w_0|| + \e^2||w_0||^2\right) \leq c_0\e. \]

For the Lipschitz estimate \eqref{NLip}, let $v,\tilde{v} \in B_{2c_0\e}(0)$ and begin with
\[ ||\mathcal{N}_{\e}(v) - \mathcal{N}_{\e}(\tilde{v})|| \leq  c\e^{-1}\left(||(\cosh s)^{-2}(v - \tilde{v})||  + \e^2||\cosh^2s (Q_0(w_0 + v) - Q_0(w_0 + \tilde{v}))|| \right). \]
Then $||(\cosh s)^{-2}(v - \tilde{v})|| \leq \e^2||v - \tilde{v}||$ as before. Now the Lipschitz estimate for the error in Proposition \ref{MeanCurvExp} gives us
\[ |Q_0(w_0 + v) - Q_0(w_0 + \tilde{v})|_{0,\alpha,s} \leq c\left(|w_0 + v|_{2,\alpha,s} + |w_0 + \tilde{v}|_{2,\alpha,s}\right) \cdot |v - \tilde{v}|_{2,\alpha,s} \mbox{ for all } s \geq S_{\e}. \]
Estimating the local norms $|\cdot|_{2,\alpha,s}$ in terms of the weighted norms it follows that
\[ |\cosh ^2s\left(Q_0(w_0 + v) - Q_0(w_0 + \tilde{v})\right)| \leq c(\cosh s)^{-2}\left(||w_0 + v|| + ||w_0 + \tilde{v}||\right)  ||v - \tilde{v}||. \]
Hence
\[ ||\cosh ^2s\left(Q_0(w_0 + v) - Q_0(w_0 + \tilde{v})\right)|| \leq c\left(||w_0 + v|| + ||w_0 + \tilde{v}||\right)  ||v - \tilde{v}|| \leq c  ||v - \tilde{v}||. \]
Putting this all together gives \eqref{NLip} and completes the proof. \qedhere
\end{proof}

\newpage

\section{A proposed gluing construction}

As discussed in the introduction, the construction of spaces of ends is an important ingredient in gluing constructions by Cauchy data matching. There is one particular gluing construction that we've had in mind, and which might be possible using our results. It follows from Lemma \ref{horgraphexplem} that the horizontal catenoids converge, in the small neck-size limit, to two horocylinders intersecting tangentially in a vertical line. This suggests that configurations of mutually tangent horocylinders may be desingularised, by perturbing each end to an end asymptotic to some horizontal catenoid. In fact, \cite{HRS} showed that if a cmc $1/2$ end is asymptotic to a horocylinder, then it lies in that horocylinder. Thus in any scheme to desingularise configurations of horocylinders, the ends of the resulting surfaces cannot be asymptotic to horocylinders, and horizontal catenoids are natural asymptotic models. 

In this final chapter we describe the gluing construction that we have pursued to produce large families of cmc $1/2$ surfaces that are symmetric with respect to a reflection in $\mathbb{H}^2 \times \{0\}$. These include examples with arbitrary positive genus $g$ and $k = g + 2$ ends. We also highlight the technical obstacles to executing this program that we've encountered.

\subsection{Outline and obstacles}

We start by outlining the approach we have pursued to produce cmc $1/2$ surfaces with genus $g > 0$ and $k = g + 2$ ends. One could also try to produce surfaces with a wider variety of topologies, for instance, with $k$ ends for any $k \geq g + 2$. However, the $k = g + 2$ case already contains the obstacles we have not been able to resolve, so we will restrict the discussion to that case.

\subsubsection*{Admissable configurations of tangent horocylinders}

We consider \emph{admissable configurations} of horocylinders, defined as follows.

\begin{Def}
A configuration of $k \geq 3$ horocylinders  
is \emph{admissable} if 
\begin{itemize}
\item [(i)] their union is a connected set, 
\item [(ii)] any two horocylinders that intersect do so tangentially, and  
\item [(iii)] every horocylinder intersects at least two others.
\end{itemize}
\end{Def}

When $k = 3$,  an admissable configuration is unique up to an isometry of $\hr$,  but when $k$ is large there are many different ones. Every admissable configuration of $k + 1$ horocylinders is constructed from a configuration of $k$ horocylinders by adding a horocylinder that intersects exactly two others. Hence, the total number of intersections in an admissable configuration of $k$ horocylinders is $3 + 2(k-3)$. 

The gluing scheme we outline desingularises a configuration by replacing each of the $k$ horocylinder ends with an end of a horizontal catenoid, and each of the $3 + 2(k-3)$ intersections with a catenoidal neck. Each time an end is added to a configuration, the desingularisation produces two closed loops around the resulting necks, and one of these may be removed without disconnecting the surface. Hence, the genus increases by one each time the number of ends increases by one, and the
 result is a cmc $1/2$ surface with $k$ ends and genus $g = k - 2$. 

\subsubsection*{Summands}

If one considers the cmc surfaces we aim to construct, each end is connected to at least two others by catenoidal necks. However, the ends we have constructed have only a single boundary component, so we require a further type of surface in the construction.
The three families of cmc $1/2$ surfaces that are required are
\begin{itemize}
\item [(a)] annular ends asymptotic to an end of a horizontal cateniod, 
\item[(b)] certain compact horizontal graphs that are close to horocylinders and have multiple boundary components,
\item [(c)] compact annuli with boundary that are close to the necks of horizontal catenoids.
\end{itemize}

We require that all of these families are symmetric with respect to a reflection in $\mathbb{H}^2 \times \{0\}$. 

The existence of the first of these families was proved in Theorem \ref{FamCMCEnd}, and it should be possible to construct the last family by similar methods. The construction of compact horizontal graphs is carried out below, in Theorem \ref{CMCGraphSlab}. For now we continue the discussion with a description of this family and the role it plays.

We propose to do the gluing in two steps. First, one would use surfaces from the families (a) and (b) to produce surfaces that have multiple boundary components, and one end asymptotic to an end of a horizontal catenoid. Second, one would glue together such pieces using the family (c). In the first step, the Cauchy data matching required to glue surfaces has presented difficulties that we have not been able to resolve. We describe our strategy to achieve the first step, and these difficulties, in what follows.


 Suppose we have an admissable configuration of horocylinders $H_1, \cdots , H_k$. For the remainder of the discussion we fix a horocylinder $H_i$. Let $\tau_i$ be the number of horocylinders that are tangent to $H_i$, so $2 \leq \tau_i \leq k - 1$. We wish to replace $H_i$ with a family of cmc ends that are asymptotic to a catenoid end and have $\tau_i$ boundary components. The family is to be parametrised by the boundary data on the $\tau_i$ components. 

To do this we first apply an isometry of $\hr$ that identifies $H_i$ with the horocylinder $y = 1$. Then we take $\Omega$ to be a domain in the $xz$ plane of the form 
\[ B_r(0,0) - \left(\bigcup_{j = 1}^{\tau_i} B_{r_j}(x_j,0)\right), \]
where $(x_j, 1)$ are the coordinates of the intersections of $H_i$ and other horocylinders, the $r_j$ are chosen so that the balls $B_{r_j}(x_j,0)$ are disjoint, and finally $r$ is chosen sufficiently large that $B_r$ contains $\cup B_{r_j}$. It will be convenient to decompose $\partial \Omega$ into the outer and inner boundaries defined as

\[ \partial \Omega_{out} := \partial B_r(0,0), \hspace{10pt} \partial \Omega_{in} := \bigcup_{j = 1}^{\tau_i} \partial B_{r_j}(x_j,0). \]
Similarly we decompose $\psi \in C^{2, \alpha}(\partial \Omega)$ into its factors on the outer and inner boundaries as $\psi = (\psi_{out}, \psi_{in})$.

Now Theorem \ref{CMCGraphSlab} below ensures that, given any sufficiently small $\psi \in C^{2, \alpha}(\partial \Omega)$, there exists a cmc $1/2$ horizontal graph on $\Omega$ which equals $1 + \psi$ on $\partial \Omega$. The boundary data $\psi$ is prescribed on a scale of $\e$ that matches that of the horizontal catenoid on $\partial \Omega_{out}$, and is slightly smaller on the $\tau_i$ components of $\partial \Omega_{in}$. Thus, the expansion \eqref{bdcurvegraph} for the boundary curve on a catenoidal end requires $\psi_{out}$ to be of the order $\e |\log(\e)| + \mathcal{O}(\e)$, hence $\psi_{in}$ should be $o(\e|\log \e|)$, and perhaps simply $\mathcal{O}(\e)$.

To produce the desired family of surfaces consider arbitrary Dirichlet data $\psi = (\psi_{in}, \psi_{out})$ on the scale just described. Fixing $\psi_{in}$ and allowing $\psi_{out}$ to vary gives a family of graphs parametrised by their boundary values on $ \partial \Omega_{out}$. If we can match the Cauchy data on this boundary with that on the boundary of an end, then we will have produced families of cmc $1/2$ surfaces with ends asymptotic to an end of a horizontal catenoid, and parametrised by boundary data $\psi_{in}$. 


\subsubsection*{Cauchy data matching}

To match Cauchy data on the outer boundary of a horizontal graph and the boundary of an end, we must find a pair of surfaces, one from each of these families, whose Dirichlet and Neumann data agree on this boundary. If the surfaces match to zeroth and first order, then the resulting surface is a continuously differentiable weak solution to a cmc $1/2$ equation, and therefore smooth by elliptic regularity. 

Let $\psi_{out}$ be as above, and $\tilde{\psi}$ be boundary data on the boundary of an end. We denote their boundary derivatives by $\partial_{\nu}\psi_{out}  $ and $\partial_{\tilde{\nu}}\tilde{\psi}$ (with compatible orientations of the boundaries). Then we are trying to show that for all small $\e$, the map
\[ \left(\psi_{out}, \tilde{\psi}\right) \mapsto \left(\psi_{out} - \tilde{\psi}, \partial_{\nu}\psi_{out} - \partial_{\tilde{\nu}}\tilde{\psi}\right) \]
has a zero. Ensuring that the first factor is zero poses no difficulty; the issue is to find data for which this holds, and the boundary derivatives agree.

To do this we need to understand the Dirichlet-to-Neumann operators for the mean curvature equations that we have solved to construct the surfaces in each family. Since we are only interested in solutions to these non-linear equations that are small and have been produced petrubatively, it should be sufficient to understand the Dirichlet-to-Neumann operator for the linearisations of the mean curvature equations. Thus far, however, we have not been able to obtain a sufficiently refined understanding of the Dirichlet-to-Neumann maps for the Jacobi operators.

\subsection{Existence of cmc $1/2$ horizontal graphs close to horocylinders.}

Recall that in the upper half space model of $\hr$, a horizontal graph is defined as a graph of the form $y = g(x,z)$. A horizontal graph has mean curvature $\frac{1}{2}$ if and only if $\mathcal{M}(g) = 1$, where the operator $\mathcal{M}(g)$ gives twice the mean curvature on the graph of $g$. It is given by (\cite{HRS})
\begin{equation}
\mathcal{M}(g) := \frac{g^2}{W^3}\left( (g^2 + g_z^2)g_{xx} - 2g_xg_zg_{xz} + (1 + g_x^2)g_{zz} + g(1 + g_x^2)\right),
\end{equation}
with 
\[ W^2 := g^2(1+g_x^2) + g_z^2. \] Note that the constant functions $g(x,z) = c > 0$ corresponding to the horocylinders satisfy $\mathcal{M}(g) = 1$. In particular we would like to consider the solution $g(x,z) = 1$ , to which the end of a horizontal catenoid converges (Lemmas \ref{horgraphexplem}, \ref{BdCurveNecksizeLim}).

A straightforward calculation shows that the linearisation of the mean curvature operator about this solution is the flat Laplacian:
\[ \left.\frac{d}{dt}\right|_{t = 0} \mathcal{M}(1 + tu) = u_{xx} + u_{zz} =: \triangle u. \]
Therefore cmc surfaces in narrow slabs can be constructed using the Implicit Function Theorem and classical results on the Laplace operator:

\begin{Thm}\label{CMCGraphSlab}
Let $\Omega$ be a smooth, bounded domain in $\R^2$. Then for all sufficiently small $\psi \in C^{2,\alpha}(\partial \Omega)$, there exists a unique solution to the boundary value problem
\[ \left\{ \begin{array}{cc} \mathcal{M}(g) = 1 & \mbox{ in } \Omega \\
 g = 1 + \psi & \mbox{ on } \partial \Omega. \end{array} \right. \]
\end{Thm}

\begin{proof}
Define $\mathcal{N}: C^{2,\alpha}(\partial \Omega) \times C^{2,\alpha}_D(\overline{\Omega}) \rightarrow C^{0,\alpha}(\overline{\Omega})$ 
by
\[ \mathcal{N}(\psi,w) := \mathcal{M}(1 + w_{\psi} + w), \]
where $w_{\psi}$ is the unique harmonic function on $\Omega$ that equals $\psi$ on $\partial \Omega$.  Now $\mathcal{N}(0,0) = \mathcal{M}(1) = 1$ and, as above, the linearisation of $\mathcal{N}$ (in $w$) about this solution is the Laplace operator
\[ D_w\mathcal{N} = \triangle : C^{2,\alpha}_D(\overline{\Omega}) \rightarrow C^{0,\alpha}(\overline{\Omega}). \]
It's a classical result that this last mapping is an isomorphism (\cite{GnT}), so the theorem follows from the Implicit Function Theorem. \qedhere
\end{proof}




\newpage

\end{document}